\documentclass[11pt]{amsart}  
\usepackage{amsmath,amsfonts,amsthm,bbm, amssymb}  
\usepackage[cmtip, all]{xy}  
\usepackage[utf8]{inputenc}  
\usepackage[english]{babel}  
  
\usepackage[ left=3cm, right=3cm, top = 2.8cm, bottom = 2.8cm ]{geometry}  
\usepackage{comment}  
\usepackage[color, notref, notcite, final]{showkeys}       
\definecolor{refkey}{gray}{.5}   % graylevel for refs  
\definecolor{labelkey}{gray}{.5} % graylevel for labels  
\definecolor{Red}{rgb}{1,0,0}

\usepackage{tikz-cd}

\usepackage{soul}  
  
\newtheorem{thm}{Theorem}[section]  
\newtheorem{prop}[thm]{Proposition}  
\newtheorem{lem}[thm]{Lemma}  
\newtheorem{cor}[thm]{Corollary}  
\newtheorem{conj}[thm]{Conjecture}  
\theoremstyle{definition}  
  
\newtheorem{defn}[thm]{Definition}  
\theoremstyle{remark}  
\newtheorem{remk}[thm]{Remark}  
\newtheorem{remks}[thm]{Remarks}  
  
\newtheorem{exm}[thm]{Example}  
\newtheorem{exms}[thm]{Examples}  
\newtheorem{notat}[thm]{Notation}  
\numberwithin{equation}{section}

{\hfill$\square$\end{defn}}  
{\hfill$\square$\end{remk}}  
{\hfill$\square$\end{remks}}  
{\hfill$\square$\end{exm}}  
{\hfill$\square$\end{exms}}  
{\hfill$\square$\end{notat}}  
  
\newcommand{\thmref}{Theorem~\ref}  
\newcommand{\propref}{Proposition~\ref}  
\newcommand{\corref}{Corollary~\ref}  
  
\newcommand{\lemref}{Lemma~\ref}

\newcommand{\sI}{{\mathcal I}}  
  
\newcommand{\sK}{{\mathcal K}}

\newcommand{\sO}{{\mathcal O}}  
  
\newcommand{\sQ}{{\mathcal Q}}  
\newcommand{\sR}{{\mathcal R}}

\newcommand{\sX}{{\mathcal X}}  
\newcommand{\sY}{{\mathcal Y}}  
\newcommand{\sZ}{{\mathcal Z}}

\renewcommand{\P}{{\mathbb P}}

\newcommand{\Z}{{\mathbb Z}}

\newcommand{\fm}{{\mathfrak m}}

\newcommand{\fM}{{\mathfrak M}}

\newcommand{\CH}{{\rm CH}}

\newcommand{\surj}{\twoheadrightarrow}  
\newcommand{\inj}{\hookrightarrow}  
\newcommand{\red}{{\rm red}}

\newcommand{\Pic}{{\rm Pic}}

\newcommand{\Spec}{{\rm Spec \,}}  
\newcommand{\sing}{{\rm sing}}

\newcommand{\divf}{{\rm div}}

\newcommand{\cyc}{{\operatorname{\rm cyc}}}

\newcommand{\et}{{\text{\'et}}}  
  
\newcommand{\ds}{{/\kern-3pt/}}

\newcommand{\ov}{\overline}

\renewcommand{\dim}{\text{\rm dim}}  
  
\newcommand{\tuborg}{\left\{\begin{array}{ll}}  
\newcommand{\sluttuborg}{\end{array}\right.}

\newcommand{\wt}{\widetilde}

\def\cO{\mathcal{O}}

\def\ol#1{\overline{#1}}

\definecolor{winered}{rgb}{0.8,0,0}  
\usepackage[pdfauthor={Federico Binda and Amalendu Krishna}, %  
				pdftitle={Rigidity for relative zero cycles} ,colorlinks=true]{hyperref}  
\hypersetup{  
	bookmarksnumbered=true,  
	linkcolor={winered},  
	citecolor=blue,  
	pagecolor=black,   
	urlcolor=black,	  
}

\newcounter{elno}     
\newenvironment{romanlist}{  
                         \begin{list}{\roman{elno})  
                                     }{\usecounter{elno}}  
                      }{  
                         \end{list}}  
                           
\newcounter{elno-abc}     
  
\newcounter{elno-abc-prime}

\begin{document}  
\title{Rigidity for relative 0-cycles}  
\author{Federico Binda and Amalendu Krishna}  
\address{Dipartimento di Matematica ``Federigo Enriques'', Università degli Studi di Milano, Via Cesare Saldini 50,    20133
Milan, Italy}  
\email{federico.binda@unimi.it}  
\address{School of Mathematics, Tata Institute of Fundamental Research,    
1 Homi Bhabha Road, Colaba, Mumbai, India}  
\email{amal@math.tifr.res.in}

\thanks{F.B.~was partially supported by the DFG SFB/CRC 1085 ``Higher Invariants'', University of Regensburg, during the preparation of this paper.}  
%\keywords{algebraic cycles, Chow groups, singular schemes}  

%\dedicatory{}  


  
\maketitle

\begin{quote}\emph{Abstract.}   
We present a relation between the classical Chow group of relative 0-cycles on  
a regular scheme $\sX$, projective and flat over an excellent Henselian  
discrete valuation ring,  
and the Levine-Weibel Chow group of 0-cycles on the special fiber.  
We show that these two Chow groups are isomorphic with finite coefficients under extra assumptions. This generalizes  
a result of Esnault, Kerz and Wittenberg.   
%We also show that the {\'e}tale cycle class map for relative 0-cycles on $\sX$  
%is an isomorphism. This generalizes a result of Saito and Sato to   
%the case of bad reduction.  
\end{quote}  
%\end{abstract}  
\setcounter{tocdepth}{1}  
%\maketitle  
\tableofcontents    
  
%\maketitle  
%\tableofcontent  
\section{Introduction}\label{sec:Intro}  
Let $A$ be an excellent Henselian discrete valuation ring with perfect residue  
field $k$ of exponential characteristic $p \ge 1$.   
Let $\sX$ be a regular scheme which is projective and flat over $A$.  
Let $X \subset \sX$ be the reduced special fiber.   
If the map $\sX \to \Spec(A)$ is an isomorphism, then   
Gabber's generalization of Suslin's rigidity theorem  
\cite{Gabber} says that the algebraic $K$-theory  
of $\sX$ and $X$ are isomorphic with coefficients prime to $p$.  
However, this rigidity theorem does not hold when the relative  
dimension of $\sX$ over $A$ is positive. One can then ask if it is possible to  
prove such an isomorphism for the higher Chow groups (which are the  
building block of $K$-theory in view of \cite{FS}) in certain bi-degrees.  
This is the context of the present work.  
  
Let $\CH_1(\sX)$ denote the classical   
Chow group \cite{Fulton} of 1-dimensional cycles on $\sX$.  
If $\sX$ is smooth over $A$ and $k$ is finite or algebraically closed,   
Saito and Sato \cite[Corollary~0.10]{SS} showed that   
there is a restriction map $\rho: {\CH_1(\sX)} \otimes_{\Z} {\Z}/{m \Z} \to   
{\CH_0(X)}\otimes_{\Z} {\Z}/{m \Z}$ which is an isomorphism, whenever  
$m$ is prime to the exponential characteristic of $k$.    
  
As part of their proof of the above restriction isomorphism,  
Saito and Sato showed that the   
{\'e}tale cycle class map for ${\CH_1(\sX)}\otimes_{\Z} {\Z}/{m \Z}$  
is an isomorphism more generally for every model $\sX\to \Spec(A)$ with   
semi-stable reduction, i.e., such that the reduced special fiber $X$   
has simple normal crossing (again, under the assumption that the residue   
field  $k$ is finite or algebraically closed).  
As an application of this, they proved that if $K$ is a local field with  
finite residue field and $Y$ is smooth and projective over $K$, then  
$\CH_0(Y)\otimes_{\Z} {\Z}/{m \Z}$ is finite, originally a conjecture due to  
Colliot-Thélène \cite{CT95}.  
  
\begin{comment}  
 No assumption on the model $\sX$ of $Y$ is   necessary for the finiteness result of Saito and Sato: in fact, in view of the étale realization isomorphism of \cite[Theorem 1.16]{SS}, the finiteness  of the group ${\CH_1(\sX)}\otimes_{\Z} {\Z}/{m \Z}$, which implies the finiteness of $\CH_0(Y)\otimes_{\Z} {\Z}/{m \Z}$, holds more generally for every family $\sX\to \Spec(A)$ with semi-stable reduction, i.e.~such that the reduced special fiber $X$ has simple normal crossing (again, under the assumption that the residue field  $k$ is finite or algebraically closed). However, in this greater generality the comparison between ${\CH_1(\sX)} \otimes_{\Z} {\Z}/{m \Z}$ and ${\CH_0(X)} \otimes_{\Z} {\Z}/{m \Z}$ no longer holds.  
\end{comment}  
  
Inspired by an argument originally due to Bloch and discussed in \cite[Appendix A]{EW}, the result of Saito and Sato  was revisited and generalized by Esnault, Kerz and Wittenberg  in \cite{EKW}. Under the assumption that the reduced special fiber $X$ is a simple  
normal crossing divisor in $\sX$, it was observed in \cite{EKW} that it is possible to replace the classical Chow group  
(see \cite{Fulton})   
$\CH_0(X)$ of the special fiber $X$ with the Friedlander-Voevodsky  
\cite{FV} motivic cohomology $H^{2d}(X, \Z(d))$, where $d = \dim_k(X)$, and still prove the existence of an isomorphism   
\begin{equation}\label{eq:restri-EKW-intro}  
 \rho: {\CH_1(\sX)} \otimes_{\Z} {\Z}/{m \Z} \to   
 H^{2d}(X,  {\Z}/{m \Z} (d)),  
 \end{equation}  
provided that some extra assumptions on $m$ or on the residue field are satisfied. This approach allowed  Esnault, Kerz and Wittenberg to generalize the restriction isomorphism of Saito and Sato by allowing the field $k$ to belong to a bigger class than just finite or algebraically  
closed fields, and the reduced special fiber to be a simple normal crossing  
divisor than just smooth. In fact, in the case of good reduction, they  
showed that $\rho$ is an isomorphism for any perfect residue field $k$.  
Note that there is always a surjective map  
 $H^{2d}(X, \Z(d)) \surj \CH_0(X)$ for a simple normal crossings divisor   
$X \subset \sX$. But this is not in general an isomorphism,  
even with finite coefficients.  
  
%  an isomorphism between the group of 1-cycles on the model and  
% In order to prove their generalization of Saito and Sato's  
% theorem,  
% Esnault, Kerz and Wittenberg had to replace the classical Chow group  
% (see \cite{Fulton})  
% $\CH_0(X)$ of the special fiber $X$ by the Friedlander-Voevodsky  
% \cite{FV} motivic cohomology $H^{2d}(X, \Z(d))$, where $d = \dim_k(X)$.  
% It is not hard to see that there is a surjective map  
% $H^{2d}(X, \Z(d)) \surj \CH_0(X)$ which is not in general an isomorphism,  
% even with finite coefficients.  
  
In this paper, we show that if we further replace the $(2d,d)$ motivic cohomology group of  
the reduced special fiber $X$ by its Levine-Weibel Chow group of  0-cycles  
\cite{LW},   
then the restriction isomorphism of   
Saito and Sato holds without any condition on $X$  
whenever the residue field is algebraically closed. More generally,  
we prove the following generalization of  
\cite{EKW} for arbitrary  perfect residue fields.  
  
We let $\sZ_1(\sX)$ denote the free abelian group on the set of integral  
1-dimensional closed subschemes of $\sX$ and let $\sZ^g_1(\sX)$ denote  
the subgroup of $\sZ_1(\sX)$ generated by integral cycles which are flat  
over $A$ and do not meet the singular locus of $X$. It follows from the  
moving lemma of Gabber, Liu and Lorenzini \cite{GLL} that the composite  
map \[\sZ^g_1(\sX) \inj \sZ_1(\sX) \surj \CH_1(\sX)\] is surjective.  
For any reduced quasi-projective scheme $Y$ over a field, let $\CH^{LW}_0(Y)$  
denote the Levine-Weibel Chow group of 0-cycles on $Y$, first introduced in \cite{LW}.  
It is a quotient of the free abelian group $\sZ_0(Y\setminus Y_{\rm sing})$ of $0$-cycles supported in the regular locus of $Y$ (see \ref{sec:Chow-grp} for a reminder of its definition).  
Let $m$ be an integer prime to the exponential characteristic of $k$ and  
let $\Lambda = {\Z}/{m\Z}$. For an abelian group $M$, write  
$M_{\Lambda} = M \otimes_{\Z} \Lambda$.  Then the following holds.

\begin{thm}\label{thm:Main-1}  
Let $\sX$ be a regular scheme which is projective and flat over an excellent   
Henselian discrete valuation ring  with perfect residue field.  
Let $X$ denote the reduced special fiber of $\sX$.   
Then there exists a commutative diagram  
\begin{equation}\label{eqn:Main-1-0}  
\begin{tikzcd}  
\sZ^g_1(\sX)_{\Lambda}  \arrow[r, "\wt{\rho}"]  \arrow[d, twoheadrightarrow]  & \sZ_0(X \setminus X_{\sing})_{\Lambda} \arrow[d, twoheadrightarrow] \\  
\CH_1(\sX)_\Lambda & \CH^{LW}_0(X)_{\Lambda} \arrow[l, "\gamma"]  
\end{tikzcd}  
\end{equation}  
such that $\gamma$ is surjective.  
\end{thm}  
Here, the vertical maps are the canonical projections, and $\wt{\rho}$ is the group homomorphism given by taking an $1$-cycle in good position and intersecting it with the reduced special fiber $X$. %Note in particular that no assumptions on $k$ are required.  
  
Let us explain how this theorem relates to the construction of \cite{EKW}. Suppose that $\sX$ has semi-stable reduction. Then by \cite[Theorem 5.1]{EKW} there exists a unique surjective homomorphism $\gamma_{EKW}$ making the diagram  
\[\begin{tikzcd}  
\sZ^g_1(\sX)_{\Lambda}  \arrow[r, "\wt{\rho}"]  \arrow[d, twoheadrightarrow]  & \sZ_0(X \setminus X_{\sing})_{\Lambda} \arrow[d, twoheadrightarrow] \\  
\CH_1(\sX)_\Lambda & H^{2d}(X, \Lambda(d)) \arrow[l, "\gamma_{EKW}"]  
\end{tikzcd}  
\]   
commutative. The group at the bottom right corner is the motivic cohomology group with $\Lambda$ coefficients (as in \eqref{eq:restri-EKW-intro}). Combining this with the cycle class map constructed in \cite{BK-1}, we obtain then a commutative diagram of surjections  
\[  
\begin{tikzcd}  
\CH_1(\sX)_\Lambda & \CH^{LW}_0(X)_{\Lambda} \arrow[l,  twoheadrightarrow,  "\gamma"] \arrow[d, twoheadrightarrow, "\cyc_X^{\mathcal{M}}"] \\  
& H^{2d}(X, \Lambda(d))  \arrow[lu,  twoheadrightarrow, "\gamma_{EKW}"]  
\end{tikzcd}  
\]  
so that we can interpret Theorem \ref{thm:Main-1} as a lift to the Levine-Weibel Chow group of the inverse restriction map considered in \cite{EKW}. Note that $\gamma_{EKW}$ exists only in the semi-stable case, while the diagram \eqref{eqn:Main-1-0} with $\gamma$ exists without assumption on the special fiber.   
  
One consequence of \thmref{thm:Main-1} is the following.  
  
\begin{cor}\label{cor:sum**}  
In the notations of \thmref{thm:Main-1}, suppose that the map  
$\wt{\rho} \colon \sZ_1^g(\sX)_{\Lambda} \to  \sZ_0(X\setminus X_{\rm sing})_{\Lambda}$   
descends to a morphism between the Chow groups  
\begin{equation}\label{eq:rho**}  
\rho \colon \CH_1(\sX)_\Lambda \to \CH^{LW}_0(X)_{\Lambda}.  
\end{equation}  
Then $\rho$ is an isomorphism. If moreover $\sX$ has semi-stable reduction, then there is a commutative diagram of isomorphisms  
\[  
\begin{tikzcd}  
\CH_1(\sX)_\Lambda \arrow[r, "\rho"] \arrow[rd] & \CH^{LW}_0(X)_{\Lambda} \arrow[d, twoheadrightarrow, "\cyc_X^{\mathcal{M}}"]\\  
&  H^{2d}(X, \Lambda(d))  
\end{tikzcd}  
\]  
\end{cor}  
The diagonal arrow in the semi-stable case agrees with the map $\rho$ of \cite[Theorem 1.1]{EKW}. We expect that the homomorphism $\rho$ in \eqref{eq:rho**} always exists, and the reason for such expectation is twofold. On one side, the Levine-Weibel Chow group is expected to be part of a satisfactory theory of cycles on singular varieties, closer to the $K$-theory of vector bundles than the cdh-motivic cohomology. The restriction homomorphism $\rho$ should then be seen as a cycle-theoretic incarnation (in certain bi-degrees) of the restriction map on $K$-groups with $\Lambda$-coefficients  
\[\iota^*\colon K_0(\sX; \Lambda) \to K_0(X; \Lambda) \]  
induced by the inclusion $\iota\colon X\hookrightarrow\sX$.   
   
The relationship between the Levine-Weibel Chow group and the $K_0$ group has been object of investigation by many authors (we recall here \cite{LW}, \cite{Levine-2}, \cite{LevineBloch}, \cite{PW}, \cite{PW-div}, \cite{KrishnaSrinivas}, \cite{Krishna3Fold} to name a few). It is known that the group $\CH_0^{LW}(X)$ can be used to detect invariants of ``additive'' type. For example, if $X$ is an arbitrary reduced curve over a field $k$, we have  
\[ \CH_0^{LW}(X)\xrightarrow{\cong} \Pic(X) \cong H^1(X, \mathbb{G}_m)\]  
generalizing the classical relationship between line bundles and Weil divisors, while   
\[ H^2(X, \Z(1)) \cong H^2(X_{\rm sn}, \Z(1)) \cong \Pic(X_{\rm sn})\]  
where $X_{\rm sn}$ denotes the semi-normalization of $X$. This is reflecting the fact that the functor $X\mapsto \Pic(X)$ considered on $\mathbf{Sch}(k)$ rather than on $\mathbf{Sm}(k)$ is not $\mathbb{A}^1$-invariant, and thus can not be captured by  an $\mathbb{A}^1$-invariant theory like Voevodsky's motivic cohomology.

%A more precise relationship between the Levine-Weibel Chow group and the emerging theory of motives with modulus (with their geometric incarnation given by the Chow groups with modulus as introduced in \cite{BS}, following the insight of Kerz and Saito \cite{KeS}) has been discussed in our work \cite{BK}.  
  
On the other side, however, with torsion coefficients prime to the exponential characteristic of $k$, there are no additive invariants to detect, and the non-$\mathbb{A}^1$-invariant theory ``collapses'' to the classical one. This statement can be made precise in the context of the theory of motives with modulus, as recently developed by Kahn Saito and Yamazaki. See \cite[Corollary 4.2.6 and Remark 4.2.7 b)]{KSY-RecandMotives} (using some results in \cite{BJAlg}). We therefore conjecture that the cycle class map   
\begin{equation}\label{eq:cycMot} \cyc_X^{\mathcal{M}}\colon \CH_0^{LW}(X)_\Lambda \to H^{2d}(X, \Lambda(d))  
\end{equation}  
is always an isomorphism with $\Lambda=\Z/m\Z$-coefficients. In a similar spirit, we expect that $\cyc_X^{\mathcal{M}}$ is an isomorphism with integral coefficients (if $k$ admits resolution of singularities, or with $\Z[1/p]$-coefficients otherwise) if the the singularities of $X$ are sufficiently mild, intuitively where additive phenomena do not occur. This is supported by the following result: if the residue field $k$ is algebraically closed and  
$X \subset \sX$ is a simple normal crossing divisor,   
it is shown in \cite{BK-1} that there is a canonical  
isomorphism   
\begin{equation}\label{eqn:Main-1-1}  
\cyc_X^\mathcal{M} \colon \CH^{LW}_0(X)\otimes_\Z[1/p] \xrightarrow{\cong} H^{2d}(X, \Z[1/p](d)),  
\end{equation}    
which holds integrally if $k$ admits resolution of singularities.   
  
In view of the above discussion, the existence of the map $\rho$ in \eqref{eq:rho**} is therefore coherent with the expectation of \cite{EKW} in the semi-stable reduction case, as explained in \cite[1]{EKW}.  
  
We are in the situation of the Corollary if we put some extra assumption.  
  
\begin{thm}\label{thm:Main-2}Let $\sX$ be as in Theorem \ref{thm:Main-1}, and assume moreover that $A$ has equal characteristic.   Then the map  $\wt{\rho}$   
	in ~\eqref{eqn:Main-1-0} descends to a morphism between the Chow groups in the following cases
	\begin{enumerate}  
%	\item  If $k$ is algebraically closed, then the map $\wt{\rho}$   
%in ~\eqref{eqn:Main-1-0} descends to a morphism between the Chow groups  
%and induces an  isomorphism of finite groups $\rho\colon  \CH_1(\sX)_{\Lambda} \xrightarrow{\cong} \CH^{LW}_0(X)_{\Lambda}$
%\begin{listabc}  
%\item  
%$\rho\colon  \CH_1(\sX)_{\Lambda} \xrightarrow{\cong} \CH^{LW}_0(X)_{\Lambda}$, and  
%\item  
%$cyc^{\et}_{\sX}: \CH_1(\sX)_{\Lambda} \xrightarrow{\cong}   
%H^{2d}_{\et}(\sX, \Lambda(d))$.  
%\end{listabc}  
\item If  $X$ has only isolated singularities and $k$ is finite. %, then the map $\wt{\rho}$   
%in ~\eqref{eqn:Main-1-0} descends to a morphism between the Chow groups if $k$ is finite. % In these cases, it induces an isomorphism of finite groups
%\[\rho\colon  \CH_1(\sX)_{\Lambda} \xrightarrow{\cong} \CH^{LW}_0(X)_{\Lambda}.\]
\item If   $\dim(X)=2$,  (with no further assumptions on the singularities of $X$). %then the map  $\wt{\rho}$   
%in ~\eqref{eqn:Main-1-0} descends to a morphism between the Chow groups (with no assumptions on the singularities of $X$). %, and induces an isomorphism 
%\[\rho\colon  \CH_1(\sX)_{\Lambda} \xrightarrow{\cong} \CH^{LW}_0(X)_{\Lambda}.\]
\end{enumerate}
In both cases, the map $\tilde{\rho}$ induces an isomorphism 
\[\rho\colon  \CH_1(\sX)_{\Lambda} \xrightarrow{\cong} \CH^{LW}_0(X)_{\Lambda},\]
and both groups are finite if $k$ is finite.
\end{thm}  
  
%\thmref{thm:Main-2}(1) %can be viewed as the most general form of the  
%restriction isomorphisms of Saito-Sato and Esnault-Kerz-Wittenberg  
%for algebraically closed residue fields. This 
%can also be viewed as  
%a rigidity theorem for the Chow groups of relative 0-cycles. It
%shows in particular that the group of $1$-cycles on $\sX$, with finite coefficients, has a particularly simple shape when $A$ is strictly Henselian, independently of how bad the special fiber is. 
%From a), it is also easy to deduce that the cycle class map \eqref{eq:cycMot} is an isomorphism in this case.  
  
If $A$ has equal characteristic, then the Gersten conjecture for Milnor $K$-theory holds, thanks to \cite{Kerz09}, and the existence of the map $\rho$ can be deduced from the validity of the Bloch-Quillen formula for singular varieties. See Section 5.3 for details (and for a comment about the assumption on the singularities of $X$ in the case of relative dimension 2).   
  
\medskip

%As recalled before, the part (b) of \thmref{thm:Main-2}(1)  
%was proven (for $k$ algebraically closed or finite) by Saito and   
%Sato \cite[Theorem~0.6]{SS} under the assumption that  
%$X$ is a simple normal crossing divisor (in particular, all its  
%irreducible components are regular) on $\sX$. The assumption of  
%$X$ being a simple normal crossing divisor was later removed by  
%Bloch \cite[Theorem~A.1]{EW} using the technique of  
%alterations due to Gabber and de Jong.
% Thanks to the usage of the  
%Levine-Weibel Chow group, the part (b) of \thmref{thm:Main-2}(1)   
%provides a direct proof of the results of Saito-Sato and Bloch for  
%algebraically closed residue fields. 
\begin{remk}\label{rem:algclosed} If the residue field $k$ is algebraically closed, the cycle class map to \'etale cohomology $\cyc_X^{\et}\colon \CH_0^{LW}(X)_\Lambda\to H^{2d}_{\et}(X, \Lambda(d))$ is an isomorphism (see \ref{lem:et-iso}). This gives in particular that the map $\tilde{\rho}$  of \eqref{eqn:Main-1-0} descends to a morphism between the Chow groups $\rho\colon \CH_1(\mathcal{X})_{\Lambda}\xrightarrow{\simeq} \CH_0^{LW}(X)_{\Lambda}$, and so $\CH_1(\mathcal{X})_\Lambda\xrightarrow{\simeq} H^{2d}_{\et}(\mathcal{X}, \Lambda(d))$ by proper base change. Note however that this isomorphism between $\CH_1(\mathcal{X})_\Lambda$ and the \'etale cohomology group  was already obtained by Bloch \cite[Theorem~A.1]{EW}, and we do not get more information in the algebraically closed field case.
\end{remk}

\begin{comment}  
If the residue field $k$ is algebraically closed and  
$X \subset \sX$ is a simple normal crossing divisor,   
it is shown in \cite{BK-1} that there is a canonical  
isomorphism   
\begin{equation}\label{eqn:Main-1-1}  
\CH^{LW}_0(X)_{\Lambda} \xrightarrow{\cong} H^{2d}(X, \Lambda(d)),  
\end{equation}  
where $\Lambda = \Z$ if $k$ admits resolution of singularities   
and $\Z[\tfrac{1}{p}]$ otherwise.  
More generally, it is shown in \cite{BK-1} that the isomorphism \eqref{eqn:Main-1-1}  
holds without any condition on $X$ if we take $\Lambda = {\Z}/m$ with  
$m$ prime to the residue characteristic.   
\thmref{thm:Main-2} therefore shows that    
the restriction isomorphism, as formulated by  
Esnault, Kerz and Wittenberg \cite{EKW} for a semi-stable reduction,  
actually holds for arbitrary reduction (when $k$ is algebraically closed).   
We expect part (1) of \thmref{thm:Main-2} to hold when $k$ is finite, but we  
do not have sufficient information about the Levine-Weibel Chow group  
at present to handle this case.  
\end{comment}

We end the introduction with a brief outline of this text.  
The proofs of our main theorems are inspired by the ideas of Esnault, Kerz  
and Wittenberg \cite{EKW}. The new insight is the introduction  
of the Levine-Weibel Chow group and its modified version from \cite{BK}  
in the picture and to show how this  
leads to the above generalizations, using the moving lemmas of  
Gabber, Liu and Lorenzini \cite{GLL}, some ideas from the Bertini  
theorems of Jannsen and Saito \cite{SS} and  a construction of  
 cycle class maps to {\'e}tale cohomology and to the Nisnevich cohomology of Milnor $K$-sheaves. These   
cycle class maps play an important role in the calculation of $\CH^{LW}_0(X)$  
with torsion coefficients.

In \S~\ref{sec:Bertini}, we discuss some forms of Bertini theorems over  
a base and in \S~\ref{sec:Lifting}, we prove our result for relative curves.  
We finish the proof of \thmref{thm:Main-1} in \S~\ref{sec:Prf-1}.  
In \S~\ref{sec:Res-iso}, we construct the cycle class maps for the   
Levine-Weibel Chow group and prove \thmref{thm:Main-2}.

\section{Bertini type theorems over a base}\label{sec:Bertini}  
In this section, we discuss some of the technical lemmas which we need in order  
to prove \thmref{thm:Main-1} when $\dim(\sX)$ is at least two. As  
some of these results are of independent interest and also used elsewhere,  
we state them separately. We fix the following general framework.  
  
\subsection{Setting}\label{sec:general-setting}  
Let $S$ be the spectrum of a discrete valuation ring $A$ with field of   
fractions $K$. Let $\eta$ be the generic point of $S$ and $s$ its closed   
point. Write $k$ for the residue field of $A$,  which is assumed to be perfect. We let $\fM = (\pi)$ denote the maximal ideal of $A$.  
Throughout this text, we fix a regular scheme   
$\sX$ which is flat and projective over $S$. We let $\phi:\sX \to S$  
be the structure morphism and let $d \ge 0$ denote the relative dimension  
of $\sX$ over $S$.  
Write $\sX_s = \sX \times_A k:= \sX \times_S \Spec(k)$   
for the special fiber of $\phi$   
and $X = (\sX_s)_{\red}\hookrightarrow \sX_s$ for the   
reduced special fiber. Given any scheme $Y$, we write   
$Y_{\rm sing}\subsetneq Y$ for the singular locus of $Y_{\rm red}$.   
In this section, we shall assume $k$ to be infinite.  
  
\begin{defn}\label{defn:hplane}    
A hyperplane $H\subset \P^N_S$ of the projective space $\P^N_S$ over $S$ is a   
closed subscheme of $\P^N_S$ corresponding to an $S$-rational point of the   
dual $(\P^N_S)^{\vee} := {\rm Gr}_S(N-1,N)$.   
\end{defn}  
  
By definition, an $S$-point of ${\rm Gr}_S(N-1,N)$ corresponds to   
(an isomorphism class of) a surjection $q\colon \sO_S^{\oplus N+1}\to \sQ$,   
where $\sQ$ is locally free (hence free, since $S$ is the spectrum of a DVR)   
of rank $N$.  
Fixing a basis $\{e_0,\ldots, e_N\}$ of $\sO_S^{\oplus N+1}$, we can write the   
kernel of $q$ as $\sum_{i=0}^N \langle a_i\rangle e_i \subset \sO_S^{N+1}$ for   
elements $a_i\in A$, not all in $\fM$. Here,  
$\langle a \rangle$ is the submodule of $\sO_S$ generated by $a\in A$.   
If $X_0,\ldots, X_n$ are the homogeneous coordinate functions on $\P^N_S$, then  
the hyperplane $H$ corresponding to $q$ is the zero locus of the linear   
polynomial $q(x) = \sum_{i=0}^{N}a_i X_i$.   
  
The same equation defines the hyperplane $H_\eta\subset \P^N_K$, the generic   
fiber of $H$. We denote by $H_s$ the hyperplane in $\P^N_k$ defined by the   
reduction of $q(x)\mod \pi$. In order to   
show the existence of good hyperplanes of $\P^N_S$, we will frequently use the   
following simple but crucial remark, due to Jannsen and Saito.  
  
\begin{lem}$($\cite[Theorem~0.1]{JannsenSaito}$)$\label{lem:Jannsen-Saito}  
Let $P$ be a projective $S$-scheme and let ${\rm sp}\colon P(K)\to P(k)$ be   
the specialization map, given by $x \mapsto \ov{\{x\}} \cap P_s$.  
Let $V_1\subset P_\eta$ and $V_2\subset P_s$ be two   
open dense subsets of $P_\eta$ and $P_s$, respectively.   
Assume that ${\rm sp}$ is surjective, $P$ has irreducible fibers and  
$P_s$ is a rational variety over $k$.   
Then the set  
\[ U := V_1(K) \cap {\rm sp}^{-1}(V_2(k))  
\]  
is non-empty.  
\end{lem}  
\begin{proof}  
This is extracted from the middle of the proof of   
\cite[Theorem 0.1]{JannsenSaito}.  
Before we give the proof, we note that if $x \in P(K)$, then  
the map $\ov{\{x\}} \to S$ must be an isomorphism and hence   
$\ov{\{x\}} \cap P_s$ is a unique closed point. In particular, the map  
${\rm sp}\colon P(K)\to P(k)$ is well-defined.  
  
Let $Z_1 = P_\eta\setminus V_1$ and $Z_2 = P_s\setminus V_2$ be the (reduced)   
closed complements of $V_1$ and $V_2$, respectively.   
Write $\ol{Z_1}$ for the closure of $Z_1$ in $P$. One clearly has that   
$Z_1(K) \subset {\rm sp}^{-1}((\ol{Z_1}\cap P_s)(k))$, so that the interesting   
set  $U$ contains ${\rm sp}^{-1}((V_2 \setminus (\ol{Z_1}\cap P_s) )(k))$.  
Since ${\rm sp}$ is surjective by assumption, it's enough to observe that   
$(V_2 \setminus (\ol{Z_1}\cap P_s) )(k)$ is non-empty.  
Now, we are given that $V_2$ is a dense open subset and  
$(\ol{Z_1}\cap P_s)$ is a proper closed subset of the irreducible scheme $P_s$.  
It follows that $V_2 \setminus (\ol{Z_1}\cap P_s)$ is open   
and dense in $P_s$. Since $k$ is infinite and $P_s$ is rational over $k$,  
one knows that $V_2 \setminus (\ol{Z_1}\cap P_s)(k)$   
is dense in $P_s$. This finishes the proof.  
\end{proof}

If we take $P = (\P^N_S)^{\vee}$, the three conditions of the Lemma are   
satisfied. Since any hyperplane $H\subset \P_N^S$ is completely determined by   
its generic fiber $H_\eta$ (as $(\P_S^N)^{\vee}(S) = (\P_K^N)^{\vee}(K)$),   
we see that the `good' hyperplanes over $S$ are parameterized by subsets of   
the form $V(K)\cap {\rm sp}^{-1}(U(k))$, for good open subsets $V$ of   
$(\P^N_K)^{\vee}$ and $U$ of $(\P^N_k)^{\vee}$, representing the prescribed   
behavior of the generic fiber and of the special fiber of $H$.  
We call a hyperplane $H$ corresponding to a $K$-rational point of a set of the   
form $V(K)\cap {\rm sp}^{-1}(U(k))$ \textit{general}.   
Our first application is the following proposition.

\begin{prop}\label{prop:Bertini-regularity}   
Let $\sX\subset \P^N_S$ be as in  \ref{sec:general-setting} such that   
$d \ge 2$. Then a general hyperplane $H\subset \P^N_S$ intersects $\sX$   
transversely, i.e., the fiber product $\sX\times_{\P^N_S} H$ is   
regular and flat $S$-scheme. If the generic fiber $\sX_\eta$ of $\sX$ is   
smooth over $K$, then $H_\eta$ is smooth as well.  
\end{prop}  
\begin{proof}  
We first note that since $\sX\to S$ is flat and both $\sX$ and $S$ are   
regular, it follows that $X = (\sX_s)_{\rm red}$ is equi-dimensional of   
dimension $d$.  
We begin by claiming that there exists an open subset $U$ of   
$(\P^N_k)^\vee$ with the dense subset $U(k)$ of $k$-rational points  
such that the following hold. Let $H$ be the hyperplane of   
$\P^N_k$ lying in $U(k)$. Then $H$ does not   
contain any component of $X$, and if $h$ denotes the image in $\sO_{X,x}$ of a   
local equation for $H$ at a closed point $x\in X$, either $h$ is a unit or   
$h \in \mathfrak{m}_{X,x}\setminus \mathfrak{m}_{X, x}^2$.  
  
It is clear that there exists a dense open subset $U'$ of   
$(\P^N_k)^\vee$ such that no hyperplane corresponding to a $k$-rational point   
of $U'$ contains any irreducible component of $X$. So we only need  
to find an open subset $U$ of   
$(\P^N_k)^\vee$ with the   
dense subset $U(k)$ such that if $H$ is the hyperplane of   
$\P^N_k$ corresponding to a point of $U(k)$ and if   
$h$ denotes the image in $\sO_{X,x}$ of a   
local equation for $H$ at a closed point $x\in X$, either $h$ is a unit or   
$h \in \mathfrak{m}_{X,x}\setminus \mathfrak{m}_{X, x}^2$.  
  
To prove this latter claim, we first assume that $k={\overline{k}}$  is separably (hence algebraically,  
since $k$ is perfect) closed. Let $W$ be the incidence   
variety $W\subset X \times (\P^N_k)^{\vee}$ consisting of points $(x, H)$ such   
that either $H$ contains a component of $X$ or $H$ does not contain any   
component of $X$ but for any local equation $h$ of $H$ at $x$,   
one has $h\in \mathfrak{m}_{X,x}^2 \subset \sO_{X, x}$. We need to estimate the   
dimension of $W$.

Let $V=H^0(\P^N_k, \cO_{\P^N_k}(1))$ be the $(N+1)$-dimensional $k$-vector space   
of linear forms, with basis $\{X_0, X_1, \ldots, X_n\}$.   
Let $x\in X$ be a closed   
point. Up to a change of coordinates, we can assume that the  hyperplane cut   
out by $X_0$ does not pass through $x$.  
We then get an isomorphism   
$V\xrightarrow{\simeq} \sO_{\P^N_k, x}/ \mathfrak{m}_{\P^N_k,x}^2$, sending   
$X_0$ to $1$. By composition, we have a surjection  
\[  
\phi_x\colon V \surj \sO_{X,x}/{\mathfrak{m}_{X,x}^2}  
\]  
and the kernel of $\phi_x$ is the $k$-vector space $V_x =   
\{H\in (\P^N_k)^\vee(k)\,|\, x\in H \ \mbox{and} \ h\in \mathfrak{m}_{X,x}^2\}$.  
Moreover, $V_x$ consists precisely of the hyperplanes which are bad   
at $x$. Notice now that we have an exact sequence of $k$-vector spaces  
\[  
0\to \mathfrak{m}_{X,x}/\mathfrak{m}_{X,x}^2\to \sO_{X,x}/\mathfrak{m}_{X,x}^2   
\to \sO_{X,x}/\mathfrak{m}_{X,x} = k \to 0.    
\]   
In particular, we get $\dim_k (\sO_{X,x}/\mathfrak{m}_{X,x}^2) \geq   
1+\dim(\sO_{X,x}) = 1+d$. Thus $\dim_k(V_x)\leq (N+1) - (d+1) = N-d$.  
  
If $W_x$ denotes the fiber at $x$ of $W$ along the first projection   
$p_1\colon W\to X\times (\P^N_k)^\vee\to X$, then we have   
$W_x = \mathbb{P}(V_x)$ and this implies from the previous estimate  
that $\dim_k(W_x) \le N-d-1$.  
Since the  projection $p_1$ is surjective, $X$ is equi-dimensional of   
dimension $d$, and for each $x\in X$, the fiber $W_x$ is a projective space of   
dimension at most $N-d-1$, we deduce that $W$ has dimension at most   
$(N-d-1) +d = N-1$.  
Since $X$ is proper over $k$, the second projection map   
$p_2\colon W\to (\P^N_k)^\vee$ is closed, hence the image is a proper closed   
subset of dimension at most $N-1$.   
We conclude that $U :=(\P^N_k)^\vee \setminus p_2(W)$ is open and dense in   
$(\P^N_k)^\vee$.   
  
Suppose now that $k$ is an arbitrary infinite perfect field  
and let $\ov{k}$ be an algebraic closure of $k$. Let $X_{\ov{k}}$ denote the  
base change of $X$ to $\ov{k}$ and let $U \subset (\P^N_{\ov{k}})^\vee$  
be the dense open subset of good hyperplanes over $\ov{k}$ obtained as  
above. Since $k$ is infinite and $(\P^N_{k})^\vee$ is rational,   
we know that the set of closed points in $(\P^N_{\ov{k}})^\vee$     
which are defined over $k$ is dense in $U$.   
Let $H \in U(k)$ be any such point. Let $x \in X$ be any  
closed point and let $h$ denote the local equation of $H$ in $\sO_{X,x}$.  
Suppose that $h$ is not a unit in $\sO_{X,x}$ so that   
$h \in \fm_{X,x}$.  
  
We know that $\pi^{-1}(x)$ is a finite set   
of closed points $\{x_1, \ldots , x_r\}$, where $\pi: X_{\ov{k}} \to X$ is the  
projection. Moreover, $H_{\ov{k}}$ has the property  
that its local equation $h$ lies in  
${\fm_{X_{\ov{k}},x_i}} \setminus {\fm^2_{X_{\ov{k}},x_i}}$ for each $i$.  
It follows that $h$ must lie in $\fm_{X,x} \setminus \fm_{X,x}^2$.  
In other words, there is an open subset $U \subset (\P^N_{k})^\vee$  
with the dense subset $U(k)$ such that every member of $U(k)$ satisfies the  
desired property. This proves the claim.

%Now, if $k$ is an arbitrary infinite perfect field, not necessarily algebraically closed, we note that the set of $k$-rational points of $(\P^N_k)^\vee$ will still have a dense intersection with $U$. In particular, $U(k)$ is (pointwise) dense in $(\P^N_k)^\vee$, as required.  
  
We now let ${\rm sp}\colon (\P^N_K)^{\vee}(K)\to (\P^N_k)^\vee(k)$ be the   
specialization map, and let $H$ be any hyperplane corresponding to a   
$K$-rational point of ${\rm sp}^{-1}(U(k))$ (note that this set is non-empty).   
Since this is a point in a projective space, say of coordinates   
$(a_0:a_1:\ldots: a_N)$, we can assume that not all the $a_i$'s are divisible   
by $\pi$. In particular, $H$ is not vertical, i.e., it is not contained in   
the special fiber $\P^N_k$. Hence it is automatically flat over $S$.  
  
Let $x\in X$ be a closed point and let $h$ be the image in $\sO_{\sX,x}$ of a   
local equation defining $\sX\cdot H = \sX\times_S H$ in a neighborhood of $x$.  
If $h$ is a unit in $\sO_{\sX,x}$, then $x\notin \sX\cdot H$ and there is   
nothing to say. Assume then that $h\in \mathfrak{m}_{\sX,x}\subset \sO_{\sX,x}$  
and write $\ov{h}$ for the image of $h$ in $\sO_{X,x}$.   
By construction, $\ov{h}$ is a local equation for $X\cdot H_s$  
and hence $\ov{h}\in \mathfrak{m}_{X,x}\setminus \mathfrak{m}^2_{X,x}$ by our  
choice of $U$. But this forces $h \in \mathfrak{m}_{\sX,x}\setminus   
\mathfrak{m}^2_{\sX,x}$ as well.  
Since $\sO_{\sX, x}$ is regular by assumption,  
this implies that $\sO_{\sX, x}/(h) = \sO_{\sX\cdot H, x}$ is a regular local   
ring. We have thus shown that every closed point    
$x\in (\sX\cdot H)_s$ has an open neighborhood in   
$\sX\cdot H$ where $\sX\cdot H$ is regular.   
Since $\sX$ is proper over $S$, these neighborhoods form a cover of   
$\sX\cdot H$, proving that $\sX\cdot H$ is regular, as required.   
  
%Since $d \ge 2$ and $X$  has the same number of connected components of $\sX$,   
%a general member of $U(k)$ has the property that  
%its intersection with each connected component of $X$ is connected. In particular,  
%${\rm sp}^{-1}(U(k))$ will have the property that the intersection of its  
%general member with $\sX$ is connected if $\sX$ happens to be connected from the beginning.  
  
For the last assertion, suppose that $\sX_\eta$ is smooth over $K$.   
In this case, the classical theorem of Bertini   
(see, for example, \cite[6.11]{Jou})   
asserts that there exists a dense Zariski open set $V\subset (\P^N_K)^\vee$   
parametrizing hyperplanes $H_\eta$ of $\P^N_K$ such that the intersection   
$\sX_\eta\cdot H_\eta$ is smooth.    
It is then enough to take $H\in V(K)\cap {\rm sp}^{-1}(U(k))$, which is   
non-empty by \lemref{lem:Jannsen-Saito}, to get a general hyperplane of   
$\P^N_S$ which satisfies all the required conditions.    
\end{proof}

\begin{remk}\label{remk:Bertini-rem}  
The proof of Proposition~\ref{prop:Bertini-regularity} gives in fact a   
bit more. In the setting of this proposition, we can consider   
the following situation. Let $(P)_s$ be any property which is generically   
satisfied by a hyperplane section of $X$ in $\P^N_k$.   
An example of such property could be `being Cohen-Macaulay'   
if $X$ is Cohen-Macaulay, or `being irreducible' if $X$ is irreducible  
(see \cite{Jou}).  
Here, generically means that the property is satisfied by each hyperplane in a   
open dense subset $V_P$ of $(\P^N_k)^\vee$.   
The set $U\cap V_P$ for the open set $U$ constructed above is then open and   
dense in $(\P^N_k)^\vee$. Thus, any hyperplane $H$ of $\P^N_S$ which   
corresponds to a $K$-rational point of ${\rm sp}^{-1}((U\cap V_P)(k))$ will   
intersect $\sX$ transversely, and its special fiber will moreover satisfy   
the property $(P)$.  
\end{remk}

We will now show that, under some extra conditions, there is a weak version of   
the Theorem of Altman and Kleiman \cite{AK} on hypersurface sections   
containing a subscheme. The proof of this fact uses a combination of  
ideas from Bloch's appendix to \cite{EW} and from \cite[Theorem~4.2]{SS}.

\begin{prop}\label{prop:Bertini-AK}   
Let $\sX\subset \P^N_S$ be as in \ref{sec:general-setting} such that  
$d \ge 2$. Let   
$Z\subset \sX$ be a regular, integral, flat relative $0$-cycle over $S$.  
Let $\sO_\sX(1)$ be the restriction of the line bundle $\sO_{\P^N_S}(1)$   
to $\sX$, and let $\sI\subset \sO_{\P^N_S}$ be the ideal sheaf of   
$Z$ in $\P^N_S$. Assume that   
%$Z\cap X_{\rm sing} = \emptyset$ and   
$Z\cap X$ is supported on one closed point $x \in X$.  
Then, for all integers $n \gg 0$ and a general section   
$\sigma \in H^0(\P^N_S, \sI(n))$,  
the hypersurface $H = (\sigma)$ defined by $\sigma$ has the  
following properties.  
\begin{enumerate}  
\item  
$\sX\cdot H$ is regular, flat and projective over $S$.  
\item  
$H\supset Z$.  
\end{enumerate}  
%If $\sX$ is connected, $H$ can be chosen so that $\sX\cdot H$ is  connected as well.  
\end{prop}  
\begin{proof}  
Let $W = Z\times_S X$ be the scheme-theoretic intersection of $Z$ with the   
reduced special fiber.   
We start by noting that the embedding dimension   
$e_x(W) := \dim_{k(x)} \mathfrak{m}_{W,x}/ \mathfrak{m}_{W,x}^2$ is at most $1$.   
Indeed, $W\subset Z$ and $Z$ is regular, finite and flat over $S$ by  
assumption. Hence $e_x(W) \leq e_x(Z)= \dim (Z) = 1$.   
As a consequence, if we let $I_{W,x}$ denote the ideal of $W$ in   
$\sO_{X,x}$, we see that     
$I_{W,x}/(I_{W,x}\cap \mathfrak{m}_{X,x}^2) \neq 0$.   
In fact, suppose that $I_{W,x} \subset \mathfrak{m}_{X,x}^2$.   
Then $\mathfrak{m}_{X,x}/( \mathfrak{m}_{X,x}^2+I_{W,x}) = \mathfrak{m}_{X,x}/   
\mathfrak{m}_{X,x}^2$ has dimension $d \ge 2$.  
But $\mathfrak{m}_{X,x}/( \mathfrak{m}_{X,x}^2+I_{W,x}) =  \mathfrak{m}_{W,x}/   
\mathfrak{m}_{W,x}^2$ has dimension at most one as shown above. This leads to  
a contradiction.

Let $\ov{\sI}$ be the ideal sheaf of $W$ in $\P^N_k$ and let   
$n \gg 0$ be any integer such that $\ov{\sI}(n)$ is generated by the  
global sections   
$V=H^0(\P^N_k, \ov{\sI}(n))\subset H^0(\P^N_k,  \sO(n))$.  
We now claim that there exists a non-empty open subset $U$ in the   
space $\mathbb{P}(V)$ such that for any $\sigma \in U(k)$, the image   
$\sigma_x$ of $\sigma$ in $\sO_{X,x}$ (for a closed point $x\in X$) is either   
a unit or an element of $\mathfrak{m}_{X,x}\setminus \mathfrak{m}_{X,x}^2$.  
  
Since the linear system associated to a basis of $V$ has base locus $W$, this   
condition is satisfied for $\sigma$ in $U'\subset V$ for each $y \neq x$ thanks   
to the  proof of Proposition~\ref{prop:Bertini-regularity}, with $U'$ open and   
non-empty. For $n \gg 0$, there is clearly another non-empty open   
$U''\subset V$ such that for $\sigma \in U''$, the restriction of $\sigma$   
has non-zero image in $I_{W,x}/(I_{W,x}\cap \mathfrak{m}_{X,x}^2)$   
(which is itself non-zero by the argument above).   
Let $U = U'\cap U''$.  
  
If $n \gg 0$, the map $a\colon H^0(\P^N_S, {\sI}(n)) \to   
H^0(\P^N_k, \ov{\sI}(n))$ is surjective.  
Then any $\sigma \in H^0(\P^N_S, {\sI}(n))$ such that $a(\sigma) \in U$ will   
satisfy the conditions of the proposition. Indeed,  
it is clear by our   
choice that $H=(\sigma)$ contains $Z$, while the regularity of   
$\sX\cdot H$ is proved exactly as in \propref{prop:Bertini-regularity}.  
\end{proof}  
  
\begin{remk}\label{remk:Hens-*}  
The reader can easily see that when $A$ is Henselian   
(which is the case for the rest of this   
text), the assumption in \propref{prop:Bertini-AK} that  
$Z\cap X$ be supported on one closed point $x \in X$, is redundant.  
\end{remk}

\section{Lifting of zero-cycles}\label{sec:Lifting}  
In this section, we shall recall the definitions of the Chow groups which  
are used in the statements of the main results.  
We shall then show how the 0-cycles on the special fiber can be  
lifted to good 1-cycles on $\sX$. Using this lifting, we shall give a  
proof of the base case of \thmref{thm:Main-1}, namely, the case of  
relative curves. This case will be used in the next section to prove  
the general case of \thmref{thm:Main-1}.  
We keep the notations of \S~\ref{sec:general-setting}.  
Throughout this section, we shall assume that the base ring   
$A$ is excellent and   
Henselian, with perfect residue field $k$ which is not necessarily infinite. % Since we can argue componentwise, we will assume until the end of the paper that $\sX$ is moreover connected.  
  
\subsection{The Chow groups of the model  and the special   
fiber}\label{sec:Chow-grp}  
Let $\sZ_1(\sX)$ be the free abelian group on the set of integral   
1-dimensional cycles in $\sX$. Let $\sR_1(\sX)$ be the subgroup of   
$\sZ_1(\sX)$ generated by the cycles which are rationally equivalent to zero   
(see, for example, \cite[\S~1]{GLL} or \cite[Chapter~20]{Fulton}).   
Let $\CH_1(\sX)=\sZ_1(\sX)/\sR_1(\sX)$ be the Chow group of 1-cycles on $\sX$   
modulo rational equivalence.   
  
We call an integral cycle $Z\in \sZ_1(\sX)$ \textit{good} if it is flat over   
$S$ and $Z\cap X_{\rm sing} = \emptyset$.   
We let $\sZ_1^g(\sX) \subset \sZ_1(\sX)$ be the free abelian group on the set   
of good cycles. In a similar spirit, we write   
$\sZ_1^{vg}(\sX)\subset \sZ_1^g(\sX)$ for the free abelian group on the set of   
integral flat $1$-cycles $Z$ which are good in the above sense and   
are regular as schemes. We call these cycles \textit{very good} on $\sX$.  
  
As $\sX$ is projective over $S$, it is an $FA$-scheme in the sense of  
\cite[2.2(1)]{GLL}. Therefore, the moving Lemma of   
Gabber, Liu and Lorenzini \cite[Theorem 2.3]{GLL} tells us that the  
canonical map  
\begin{equation}\label{eqn:GLL*}  
\frac{\sZ_1^g(\sX)}{\sZ_1^g(\sX)\cap \sR_1(\sX)} \to \CH_1(\sX)  
\end{equation}  
is an isomorphism.   
In other words, every cycle $\alpha \in \CH_1(\sX)$ has a representative   
$\alpha = \sum_{i=1}^n n_i [Z_i]$ with each $Z_i$ a good integral cycle.  
This will play a crucial role in the proofs of our main results.  
  
We now recall the definition of Levine-Weibel Chow group of 0-cycles on $X$ from   
\cite{LW} and its modified version from \cite{BK}.  
Let $X_{\rm reg}$ denote the disjoint union of the smooth loci of the   
$d$-dimensional irreducible components of $X$.  
A regular (or smooth) closed point of $X$ will mean a closed point lying in  
$X_{\rm reg}$.  
Let $Y \subsetneq X$ be a closed subset not containing any $d$-dimensional  
component of $X$ such that $X_{\rm sing} \subseteq Y$.   
Let $\sZ_0(X,Y)$ be the free abelian group on closed points of $X \setminus Y$.  
We shall often write $\sZ_0(X,X_{\rm sing})$ as $\sZ_0(X)$.  
  
\begin{defn}\label{defn:0-cycle-S-1}  
Let $C$ be a reduced scheme which is of pure dimension one over $k$.  
We shall say that a pair $(C, Z)$ is \emph{a good curve  
relative to $X$} if there exists a finite morphism $\nu\colon C \to X$  
and a  closed proper subscheme $Z \subsetneq C$ such that the following hold.  
\begin{enumerate}  
\item  
No component of $C$ is contained in $Z$.  
\item  
$\nu^{-1}(X_{\rm sing}) \cup C_{\rm sing}\subseteq Z$.  
\item  
$\nu$ is local complete intersection at every   
point $x \in C$ such that $\nu(x) \in X_{\rm sing}$.   
\end{enumerate}  
\end{defn}

Let $(C, Z)$ be a good curve relative to $X$ and let   
$\{\eta_1, \cdots , \eta_r\}$ be the set of generic points of $C$.   
Let $\sO_{C,Z}$ denote the semilocal ring of $C$ at   
$S = Z \cup \{\eta_1, \cdots , \eta_r\}$.  
Let $k(C)$ denote the ring of total  
quotients of $C$ and write $\sO_{C,Z}^\times$ for the group of units in   
$\sO_{C,Z}$. Notice that $\sO_{C,Z}$ coincides with $k(C)$   
if $|Z| = \emptyset$.   
As $C$ is Cohen-Macaulay, $\sO_{C,Z}^\times$  is the subgroup of group of units in the ring of total quotients $k(C)^\times$   
consisting of those $f\in \sO_{C,x}$ which are regular and invertible for every $x\in Z$ (see \cite{EKW}, Section 1 for further details).   
  
Given any $f \in \sO^{\times}_{C, Z} \inj k(C)^{\times}$, we denote by    
${\rm div}_C(f)$ (or ${\rm div}(f)$ in short)   
the divisor of zeros and poles of $f$ on $C$, which is defined as follows. If   
$C_1,\ldots, C_r$ are the irreducible components of $C$,   
and $f_i$ is the factor of $f$ in $k(C_i)$, we set   
${\rm div}(f)$ to be the $0$-cycle $\sum_{i=1}^r {\rm div}(f_i)$, where   
${\rm div}(f_i)$ is the usual   
divisor of a rational function on an integral curve in the sense of   
\cite{Fulton}. As $f$ is an invertible   
regular function on $C$ along $Z$, ${\rm div}(f)\in \sZ_0(C,Z)$.

By definition, given any good curve $(C,Z)$ relative to $X$, we have a   
push-forward map $\sZ_0(C,Z)\xrightarrow{\nu_{*}} \sZ_0(X)$.  
We shall write $\sR_0(C, Z, X)$ for the subgroup  
of $\sZ_0(X)$ generated by the set   
$\{\nu_*({\rm div}(f))| f \in \sO^{\times}_{C, Z}\}$.   
Let $\sR_0(X)$ denote the subgroup of $\sZ_0(X)$ generated by   
the image of the map $\sR_0(C, Z, X) \to \sZ_0(X)$, where  
$(C, Z)$ runs through all good curves relative to $X$.  
We let $\CH^{BK}_0(X) = \frac{\sZ_0(X)}{\sR_0(X)}$.  
  
If we let $\sR^{LW}_0(X)$ denote the subgroup of $\sZ_0(X)$ generated  
by the divisors of rational functions on good curves as above, where  
we further assume that the map $\nu: C \to X$ is a closed immersion,  
then the resulting quotient group ${\sZ_0(X)}/{\sR^{LW}_0(X)}$ is  
denoted by $\CH^{LW}_0(X)$. Such curves on $X$ are called the   
{\sl Cartier curves}. There is a canonical surjection  
$\CH^{LW}_0(X) \surj \CH^{BK}_0(X)$. The Chow group $\CH^{LW}_0(X)$ was  
discovered by Levine and Weibel \cite{LW} in an attempt to describe the  
Grothendieck group of a singular scheme in terms of algebraic cycles.  
The modified version $\CH^{BK}_0(X)$ was introduced in \cite{BK}.

We remark here that the definition of $\CH^{LW}_0(X)$ given above is  
mildly different from the one given in \cite{LW} because we do not  
allow non-reduced Cartier curves.   
However, it does agree  
with the definition  of \cite{LW}   
if $k$ is infinite by \cite[Lemmas~1.3, 1.4]{Levine-2}.  Note that over finite fields the situation is unclear (but see \cite{BKS} for the case of surfaces), since the standard norm trick to reduce to the case of infinite fields for comparison does not work for the Levine-Weibel Chow group. The situation is substantially better if one uses its variant \cite{BK} instead.

\subsection{Lifting 0-cycles on the special fiber  to 1-cycles on $\sX$}  
\label{sec:Lifting-prf}  
From the above definitions of $\CH_1(\sX)$ and $\CH^{LW}_0(X)$, it  
is not clear if the 1-cycles on $\sX$ always restrict to admissible  
0-cycles on  
$X$, nor if the restriction (whenever defined) preserves the rational equivalence. This question will be  
addressed in the next section. Here, we solve the reverse problem, namely,  
we show that the Levine-Weibel 0-cycles on $X$ can be lifted to good   
1-cycles on $\sX$, following the idea of \cite{EKW}. Using this lifting,  
we shall prove \thmref{thm:Main-1}.   
We fix an integer  $m$ prime to the exponential characteristic of $k$  
and let $\Lambda = {\Z}/{m\Z}$.  
For an abelian group $M$, we let $M_{\Lambda} = M \otimes_{\Z} \Lambda$.  
  
Let $[Z]\in \sZ_1^g(\sX)$ be an integral good 1-cycle.   
Intersecting $[Z]$ with the reduced special fiber $X$ gives rise to a   
0-cycle $[Z\cap X]$, which is supported in the regular locus of $X$.   
Here, $[Z\cap X]$ is the 0-cycle in $\sZ_0(X)$ associated to the  
(possibly non-reduced) 0-dimensional scheme-theoretic intersection  
$Z \cap X$.  
This gives rise to the {\sl restriction} homomorphism  
on the cycle group  
\begin{equation}\label{eq:restriction-map-generators}  
\wt{\rho}: \sZ_1^g(\sX)\to \sZ_0(X, X_{\rm sing}),   
\quad [Z]\mapsto [Z\cap X].  
\end{equation}  
  
To prove \thmref{thm:Main-1}, we begin by recalling the following result.   
The proof is classical, and in   
this form is essentially taken from \cite{EKW}.   
We review the proof in order to fix our notation.

\begin{prop}$($\cite[\S~4]{EKW}$)$\label{prop:rho-is-onto}  
Given a regular closed point $x \in X$,   
there exists an integral $1$-cycle $Z_x\subset \sX$ which is regular,   
finite and flat over $S$ such that $Z_x\times_S X = \{x\}$   
scheme-theoretically. In particular,   
the restriction map $\wt{\rho}$ of \eqref{eq:restriction-map-generators} is   
surjective.  
\end{prop}  
\begin{proof}  
Let $x\in X_{\rm reg}$ be a closed point and let $\sO_{\sX, x}$ be the local   
ring of $\sX$ at $x$. Since $\sX$ is regular, $\cO_{\sX, x}$ is a regular   
local ring. In particular, it is a unique factorization domain.   
There is then a prime element $\sigma\in \mathfrak{m}_{\sX, x} \setminus   
\mathfrak{m}_{\sX, x}^2$ and an integer $n>0$ such that   
$\sigma^n = \pi c$, where $\pi\in \sO_{\sX,x}$ is the uniformizer of $A$ and $c$ is a unit.  
Indeed, $\pi$ can not be a product of distinct prime elements in   
$\sO_{\sX,x}$, since $\sO_{\sX, x}\otimes_A (A/(\pi)) = \cO_{\sX_s, x}$ has a   
unique minimal prime (its reduction, $\sO_{X, x}$ is a regular local ring).  
We can now complete $\sigma$ to a regular sequence   
$(\sigma, a_1,\ldots, a_d)$ generating the maximal ideal   
$\mathfrak{m}_{\sX, x}$ such that the images   
$(\ov{a}_1, \ldots, \ov{a}_d)$ in $\sO_{X, x} = \sO_{\sX,x}/(\sigma)$ form a   
regular sequence, generating the maximal ideal $\mathfrak{m}_{X, x}$.   
  
Let $\Spec({\sO_{\sX,x}/(a_1,\ldots, a_d)})$ be the closed subscheme of   
$\Spec ({\sO_{\sX,x}})$ associated to the ideal $(a_1,\ldots, a_d)$.  
It is clearly integral, regular, local, $1$-dimensional and flat over $S$.   
If we let $\wt{Z}_x$ denote its closure in $\sX$, then $\wt{Z}_x$ is   
projective and dominant of   
relative dimension zero over $A$. In particular, it is  
finite and flat over $S$. We can therefore write $\wt{Z}_x = \Spec(B)$.  
  
Since $S$ is Henselian, the finite $A$-algebra $B$ is totally split.   
Hence, there is a unique irreducible component $Z_x$ of $\wt{Z}_x$ such that   
$x\in Z_x$.  
The scheme $Z_x$ is then regular because its local ring at the unique closed   
point $x$ agrees with ${\sO_{\sX,x}}/(a_1,\ldots, a_d)$.  
Furthermore, $Z_x$ is integral, finite and flat over $S$ with  
$Z_x\times_S X = \{x\}$.  
\end{proof}  
  
Note that thanks to the proposition above, we have in fact shown that the   
composite map  
\[  
\sZ_1^{vg}(\sX) \inj \sZ_1^{g}(\sX) \xrightarrow{\wt{\rho}}   
\sZ_0(X, X_{\rm sing})  
\]  
is surjective.   
  
\subsection{The case of relative dimension one}\label{sec:dim-1}  
We continue with the assumption that $A$ is Henselian and $k$ is perfect  
(but not necessarily infinite).  
Suppose that $\dim_S(\sX)=1$ so that $\sX$ is a family of projective  
\textit{curves} over $S$. We shall now give the proof of \thmref{thm:Main-1}  
in this case.  
   
Since $X$ is reduced by construction, we have by   
\cite[Lemma~3.12]{BK}, the canonical isomorphisms   
$\CH_0^{LW}(X) \xrightarrow{\simeq} \CH^{BK}_0(X)   
\xrightarrow{\simeq} \Pic(X) \cong H^1_{\et}(X, \mathbb{G}_m)$.  
As a scheme, $\sX$ is integral and purely two-dimensional so that we can   
identify $\CH_1(\sX)$ with $\CH^1(\sX)$.  
Since $\sX$ is moreover separated, regular (hence locally factorial) and   
Noetherian, there are classical isomorphisms   
$\CH^1(\sX)  \xrightarrow{\simeq} \Pic(\sX) \cong H^1_{\et}(\sX, \mathbb{G}_m)$.   
Tensoring these groups with $\Lambda = \Z/m$, the Kummer sequence gives us    
injections   
\[  
\CH^1(\sX)_\Lambda  \xrightarrow{\cong}  H^1_{\et}(\sX, \mathbb{G}_m)_{\Lambda}    
\hookrightarrow H^2_{\et}(\sX, \Lambda(1))   
\]  
\[  
\CH^{LW}_0(X)_\Lambda  \xrightarrow{\cong}  H^1_{\et}(X, \mathbb{G}_m)_{\Lambda}   
\hookrightarrow H^2_{\et}(X, \Lambda(1)).  
\]  
  
Using these injections, we get a diagram of solid arrows  
\begin{equation}\label{eqn:case-of-curves}  
\xymatrix@C.8pc{    
& \sZ_1^g(\sX)_{\Lambda} \ar@{->>}[r]^-{\wt{\rho}} \ar@{->>}[d]  
\ar@{->>}[dl]_-{\alpha_{\sX}}   
& \sZ_0(X, X_{\rm sing})_{\Lambda} \ar@{->>}[d] \ar@{-->}[dl]   
\ar@{->>}[dr]^-{\alpha_X} & \\  
\CH_1(\sX)_{\Lambda} \ar[r]^-{\cong} \ar[dr]_{cyc^{\et}_{\sX}} & \  
\Pic(\sX)_{\Lambda} \ar@{^{(}->}[d]   
\ar[r]^-{\rho} & \Pic(X)_{\Lambda} \ar@{^{(}->}[d] & \CH^{LW}_0(X)_{\Lambda}   
\ar[dl]^-{cyc^{\et}_X}    
\ar[l]_-{\cong} \\  
& H^2_{\et}(\sX, \Lambda(1)) \ar[r]^-{\cong} & H^2_{\et}(X, \Lambda(1)). &}  
\end{equation}  
  
All horizontal arrows in the middle are induced by the restriction to the  
reduced special fiber. In particular, the two squares in the middle   
are commutative. The two triangles on the top left and top right  
can be easily seen to be commutative by recalling the construction of the   
isomorphism between the  
Picard group and the Chow group of codimension one cycles.  
The two triangles on the bottom left and bottom right commute by the  
definition of the cycle class maps to {\'e}tale cohomology.  
  
The bottom horizontal arrow ~\eqref{eqn:case-of-curves}   
is an isomorphism by the rigidity theorem  
for {\'e}tale cohomology (a consequence of the proper base change theorem,   
see \cite[Chapter~VI, Corollary~2.7]{Milne}).   
The top horizontal arrow is surjective by   
Proposition~\ref{prop:rho-is-onto}. From the commutativity of   
\eqref{eqn:case-of-curves}, we immediately see that the canonical map   
$\alpha_{\sX}\colon \sZ_1^g(\sX)_{\Lambda} \surj \CH_1(\sX)_{\Lambda}$ factors via   
$\wt{\rho}$. Equivalently, we have ${\rm Ker}(\wt{\rho}) \subseteq   
{\rm Ker} (\alpha_{\sX})$.   
This gives the dashed arrow $\wt{\gamma}\colon \sZ_0(X, X_{\rm sing})_\Lambda \to   
\CH_1(\sX)_\Lambda$, which is automatically surjective.

A second inspection of \eqref{eqn:case-of-curves},   
using this time the fact that   
$\CH_1(\sX)_\Lambda \to H^2_{\et}(\sX, \Lambda(1))$ is injective,   
shows similarly that   
${\rm Ker} (\alpha_X) \subseteq {\rm Ker} (\wt{\gamma})$.  
Combining all this, we finally get a  
surjective group homomorphism $\gamma$ fitting in the   
commutative diagram  
\begin{equation}\label{eqn:case-of-curves-0}  
\xymatrix@C.8pc{    
\CH_1(\sX)_{\Lambda}  \ar@{^{(}->}[d] & \CH^{LW}_0(X)_{\Lambda}   
\ar[l]_-{\gamma}  \ar@{^{(}->}[d] \\  
H^2_{\et}(\sX, \Lambda(1)) \ar[r]^{\cong} & H^2_{\et}(X, \Lambda(1)).}  
\end{equation}   
  
We also deduce from ~\eqref{eqn:case-of-curves-0} that $\gamma$ has to be   
injective as well. Since $\gamma$ is clearly an inverse of the map   
$\wt{\rho}$ on the generators, we have then shown the following result  
which proves \thmref{thm:Main-1} and a general form of the  
part (1) of \thmref{thm:Main-2} for curves.  
  
\begin{prop}\label{prop:lifting-for-curves}  
Let $A$ be an excellent Henselian discrete valuation ring with  perfect residue  
field.  
Let $\sX$ be a regular scheme, flat and projective over $S$ of relative   
dimension one. Then the restriction homomorphism $\wt{\rho}$ of    
\eqref{eq:restriction-map-generators}    
induces an isomorphism  
\[  
\rho \colon \CH_1(\sX)_{\Lambda} \xrightarrow{\cong} \CH^{LW}_0(X)_\Lambda.  
\]  
\end{prop}  
  
\section{Proof of \thmref{thm:Main-1}}\label{sec:Prf-1}  
We shall now prove \thmref{thm:Main-1} using the Bertini theorems of  
\S~\ref{sec:Bertini} and the lifting proposition of \S~\ref{sec:Lifting}.  
We assume $A$ to be an excellent Henselian discrete valuation ring with   perfect  
residue field $k$. The rest of the assumptions and notations are same as  
in \S~\ref{sec:general-setting}.   
%We assume that $\dim_S \sX \geq 1$.  
  
\subsection{Factorization of $\alpha_{\sX}$ via $\wt{\rho}$}  
\label{sec:Factor-rho}  
We begin by showing the first part of Theorem~\ref{thm:Main-1},   
i.e., we show that the canonical surjection $\alpha_{\sX}:  
\sZ^g_1(\sX)_{\Lambda} \surj \CH_1(\sX)_{\Lambda}$ factors through $\wt{\rho}$.  
This is a consequence of the following result, whose proof goes   
through the steps of \cite[Proposition~4.1]{EKW}, using   
Proposition~\ref{prop:Bertini-AK} instead of the Bertini Theorem of   
Jannsen-Saito proved in \cite{SS}.

\begin{prop}$($\cite[Proposition~4.1]{EKW}$)$\label{prop:can-lifting}   
Let $Z \in \sZ_1^g(\sX)$ be a good, integral 1-cycle and let   
$n[x] = [Z \cap X]$ for some $x \in X_{\rm reg}$ and $n>0$.   
Then $\alpha_{\sX}( Z - n Z_x) = 0$ in $\CH_1(\sX)_{\Lambda}$,  
where $Z_x$ is as in Proposition~\ref{prop:rho-is-onto}.  
\end{prop}  
\begin{proof}  
By the standard pro-$\ell$-extension argument, we can assume that the   
residue field of $A$ is infinite.   
The proof is now by induction on the relative dimension of $\sX$ over $S$.  
The case $d = 0$ is trivial and    
the case $d=1$ is provided by Proposition~\ref{prop:lifting-for-curves}.  
We now assume that $d \ge 2$.  
  
Assume first that $Z$ is regular as well.   
The general case will be treated later, using a trick due to   
Bloch \cite[Appendix~A]{EW}).   
By an iterated application of Proposition~\ref{prop:Bertini-AK},   
we can find   
\begin{enumerate}  
\item  
a hypersurface section $H$ of $\sX$ which is regular, flat and projective   
over $S$ such that $Z \subset H$, and   
\item  
a relative curve $H'$ over $S$ (i.e., $\dim_S H'=1$) which is regular, flat   
and projective over $S$ and contains $Z_x$.   
\end{enumerate}  
  
We can also assume that $Z'':=H'\cap H$ is regular as well, and that   
$H'\cap H \cap X$ consists only of the reduced point $x$.  
Note that we can do this since $x \in X$ is in the regular locus of   
$X$, so that we can choose $H'$ and $H$ which meets transversely there.  
  
By our induction hypothesis, we have that $\alpha_H(Z-nZ'') = 0$ in   
$\CH_1(H)_{\Lambda}$. Moreover, it follows from  
\propref{prop:lifting-for-curves} that $\alpha_{H'}(Z''-Z_x) =0$ in   
$\CH_1(H')_{\Lambda} =\Pic(H')_{\Lambda}$.   
In particular, we get $n \alpha_{H'}(Z''-Z_x) =0$.   
But then, we get   
\[   
\alpha_{\sX} ( Z-nZ_x) = (\iota_H)_* (\alpha_H( Z-nZ'')) +   
(\iota_{H'})_* (\alpha_{H'}(nZ''-nZ_x)) =0  
\]  
in $\CH_1(\sX)_{\Lambda}$, as required.  
Here, $\iota_H$ (resp. $\iota_{H'}$) is the inclusion  
$H \inj \sX$ (resp. $H' \inj \sX$).  
  
Suppose now that $Z$ is not necessarily regular.  
Following an idea of Bloch, we let   
$Z^N$ be the normalization of $Z$. Since $A$ is excellent and  
$Z$ is finite over $A$ (as it is a good 1-cycle),   
the map $Z^N\to Z$ is finite. In particular, there is a factorization  
\[   
\xymatrix@C.8pc{   
Z^N \ar@{^{(}->}[r] \ar[d] & \P^M_{\sX} \ar[d]^q \\   
Z \ar@{^{(}->}[r] & \sX,}  
\]  
where $q$ is the canonical projection.  
We are then reduced to prove the statement in $\P^M_{\sX}$    
for $Z^N$ and any regular lift of $Z_x$ to $\P^M_{\sX}$, chosen so that it contains $Z^N \cap \P^M_X$.
%Such a regular lift of $Z_x$ can be obtained by choosing a rational section   
%$\sX \hookrightarrow \P^M_{\sX}$ of the projection $q$ and taking the image of   
%$Z_x$ under this section.  
Since $Z^N$ is now regular, the claim follows from the previous case.  
\end{proof}  
  
An immediate consequence of \propref{prop:can-lifting} is the following.   
  
\begin{cor}\label{cor:Factor}  
The lifting of 0-cycles of Proposition~\ref{prop:rho-is-onto} gives rise to  
a well-defined group homomorphism  
$\wt{\gamma}: \sZ_0(X, X_{\rm sing})_{\Lambda} \to \CH_1(\sX)_{\Lambda}$  
such that the diagram  
\begin{equation}\label{eqn:Factor-0}  
\xymatrix@C.8pc{  
\sZ_1(\sX)_{\Lambda} \ar[r]^-{\wt{\rho}} \ar[dr]_-{\alpha_{\sX}} &  
\sZ_0(X, X_{\rm sing})_{\Lambda} \ar[d]^-{\wt{\gamma}} \\  
& \CH_1(\sX)_{\Lambda}}  
\end{equation}  
commutes.  
\end{cor}  
\begin{proof}  
Let $Z \in \sZ_1^g(\sX)$ be a good, integral 1-cycle. Since $A$ is Henselian  
and $Z$ is finite over $A$, the intersection $Z \cap X$ must be supported  
on a (regular) closed point, say, $x \in X$. In particular, we must have  
$[Z \cap X] = n[x]$ for some integer $n > 0$.   
Now, it follows from \propref{prop:can-lifting} that  
\[  
\alpha_{\sX}([Z]) - \wt{\gamma} \circ \wt{\rho}([Z]) =  
\alpha_{\sX}([Z] - n[Z_x]) = 0  
\]  
and this proves the corollary.    
\end{proof}

\subsection{Factorization of $\wt{\gamma}$ through rational   
equivalence}\label{sec:Factor-gamma}  
Now that we have constructed the map $\wt{\gamma}$ at the level of the  
cycle groups, our next goal is to show that it factors through the   
cohomological Chow group   
$\CH^{LW}_0(X)_{\Lambda}$ of the reduced special fiber $X$.  
In fact, we shall show (probably) more in the sense that    
$\wt{\gamma}$ actually has a factorization  
\begin{equation}\label{eqn:gamma-BK}  
\wt{\gamma}: \sZ_0(X, X_{\rm sing})_{\Lambda} \surj \CH^{LW}_0(X)_{\Lambda}  
\surj \CH^{BK}_0(X)_{\Lambda} \to \CH_1(\sX)_{\Lambda}.  
\end{equation}

As we will see below, apart from giving us a stronger statement,  
the approach of working with $\CH^{BK}_0(X)$ also allows  
us to simplify the Cartier curves that give  
relations in $\sR^{LW}_0(X)$ which we want to kill  
in $\CH_1(\sX)_{\Lambda}$. It allows us to assume that the Cartier curves  
are regularly embedded in $X$. This is an essential requirement in our  
proof.

It is not known if the canonical map $\CH^{LW}_0(X) \surj \CH^{BK}_0(X)$ is  
an isomorphism in general. We refer to \cite[Theorem~3.17]{BK}  
for some positive results.

We shall closely follow the proof of \cite[Theorem~5.1]{EKW}   
(and we keep similar notations for the reader's convenience),   
with one simplification and one complication.  
The simplification is that using the Levine-Weibel Chow group (or, rather, its variant introduced in \cite{BK}),  
we don't have to deal with the   
``type-$1$'' relations (see \cite[\S~2.2]{EKW}),   
arising from the relations in the Suslin homology group $H_0^{S}(X_{\rm reg})$.  
On the other hand, the complication is that without any assumption on the   
geometry of $X$, we have to consider arbitrary l.c.i. curves $C$   
(and not simply SNC subcurves in $X$ as in \textit{loc.cit.}).  
Note that these l.c.i. curves may not even be embedded inside $X$.   
In order to lift our complicated relations in $X$ to the model $\sX$,   
we shall use the argument of \cite[Lemma 2.5]{GLL}.

We will need the following commutative algebra Lemma whose proof can be  
obtained from \cite[Theorem~16.3]{Matsumura}.  
  
\begin{lem}\label{lem:regularsequences}  
Let $R$ be a Noetherian local ring and let $I\subset R$ be an ideal    
generated by a regular sequence   
$a_1,\ldots, a_n$. Let $b_1,\ldots, b_n\in I$ be elements such that the image   
of $\{b_1,\ldots, b_n\}$ in $I/I^2$ is a basis over $R/I$.   
Then $b_1,\ldots, b_n$ is a regular sequence in $R$.  
\end{lem}

\begin{prop}\label{prop:Factor-gamma-*}  
The lifting map $\wt{\gamma}\colon \sZ_0(X, X_{\rm sing})_\Lambda \to   
\CH_1(\sX)_\Lambda$ of \corref{cor:Factor}  
factors through  $\CH^{BK}_0(X)_\Lambda$.  
\end{prop}  
\begin{proof}  
Since the case of relative dimension one is already shown in  
\S~\ref{sec:dim-1}, we shall assume that $d = \dim_S(\sX) \ge 2$.  
We need to show that for any good curve $\nu\colon C\to X$ in the sense of   
Definition~\ref{defn:0-cycle-S-1} and any rational function $f$ on $C$   
which is regular along $\nu^{-1}(X_{\rm sing})$,   
we have $\wt{\gamma}(\nu_*({\rm div}(f)))=0$ in $\CH_1(\sX)_\Lambda$.  
  
We will first show that this relation holds when the curve   
$C$ is regularly embedded inside $X$ (i.e., when the morphism   
$\nu$ is a regular closed embedding).   
The general case will be handled by factoring $\nu$ as a regular closed   
embedding $C\hookrightarrow \P^N_{X}$ followed by the projection   
$\P^N_X\to X$, and using the fact that the Chow groups   
$\CH_1(\sX)$ and $\CH^{BK}_0(X)$ admit proper   
push-forward for smooth morphisms.  
  
So, let $C\hookrightarrow X$ be such an embedded l.c.i. curve.    
Write $C_\infty$ for the finite set of points    
$(C\cap X_{\rm sing}) \cup \{\eta_1,\ldots, \eta_r \}$, where each   
$\eta_i$ is a    
generic point of $C$ and $C\cap X_{\rm sing}$ denotes the set of closed points   
of the intersection of $C$ with $X_{\rm sing}$.   
Let $\sO_{X, C_\infty}$ be the semi-local ring of $X$ at   
$C_{\infty}$ and let $I_{C, C_{\infty}}$ be the ideal of $C$ in   
$\sO_{X, C_{\infty}}$ so that $\sO_{C, C_\infty} =   
{\sO_{X, C_\infty}}/{I_{C, C_{\infty}}}$.  
By definition, $C$ is regularly embedded at each  point   
$x \in C\cap X_{\rm sing}$, and it is regularly embedded at the generic points.   
Hence, as a module over $\sO_{C, C_\infty}$, the conormal sheaf   
$I_{C,C_\infty}/ I_{C,C_\infty}^2$ admits a free set of generators,   
given by the image in $I_{C, C_\infty}/I_{C, C_\infty}^2$ of a regular sequence   
$a_1, \ldots, a_{d-1}$ in $\cO_{X, C_\infty}$.   
  
We shall inductively modify the sequence $a_1, \ldots, a_{d-1}$   
(without changing the induced basis of $I_{C,C_\infty}/ I_{C,C_\infty}^2$) in   
order to construct a good lifting of $C$ to the model $\sX$,   
following the recipe of \cite[Lemma 2.5]{GLL}.   
First, we note that according to Definition~\ref{defn:0-cycle-S-1},   
the curve $C$ is not contained in $X_{\rm sing}$. By a moving argument, we can   
also assume that $C$ does not contain any component of $X_{\rm sing}$.  
Indeed, the Cartier condition of $C$ implies that it will contain a component   
of $X_{\rm sing}$ only if $\dim (X_{\rm sing}) =0$.   
On the other hand, in this latter case, we can use a moving  
argument to ensure that $C$ does not hit $X_{\rm sing}$  
(see \cite[Lemma~1.3]{ESV}).  
Thus, the ideal $I_{C, C_\infty}$ of $\sO_{X, C_{\infty}}$ does not contain, and it   
is not contained in the localization of any minimal prime   
$\mathfrak{p}$ of $X_{\rm sing}$ in $\sO_{X, C_\infty}$.  
   
Up to possibly adding an element of $I_{C, C_\infty}^2$ to   
$a_1 \in I_{C, C_\infty} \subset \sO_{X, C_\infty}$, we can now choose   
$\hat{a}_1 \in  \sO_{\sX, C_\infty}$, lifting $a_1$, with the property that   
$\hat{a}_1$ does not belong to any minimal prime of   
$X_{\rm sing}$ in $\sO_{\sX, C_\infty}$.   
In other words, $V(\hat{a}_1)$ in $\Spec(\sO_{\sX, C_\infty})$ does not contain   
any irreducible component of $X_{\rm sing}$. Moreover, each irreducible  
component of $V(\hat{a}_1)$ has codimension exactly one in   
$X_{\rm sing} \times_{\sX} \Spec(\sO_{\sX, C_\infty})$ with the reduced   
induced closed subscheme structure of $X_{\rm sing}$.  
Note that thanks to Lemma~\ref{lem:regularsequences}, the modification by   
adding elements of $I_{C, C_\infty}^2$ gives another regular sequence defining   
$I_{C, C_\infty}$.  
  
We now fix $\hat{a}_1$ and $a_1$ chosen above, and proceed.  
Since locally $V(\hat{a}_1)\cap C = C$ in $\Spec(\sO_{X, C_\infty})$, the ideal   
$I_{C, C_\infty}$ is not contained in any minimal prime of   
$V(\hat{a}_1)\cap X_{\rm sing}$.   
Thus, we can alter $a_2$ by an element of $I_{C, C_\infty}^2$ so that we can   
assume that $a_2$ in particular is not in any  minimal prime of   
$V(\hat{a}_1)\cap X_{\rm sing}$.  
We now lift $a_2$ to $\hat{a}_2\in \sO_{\sX, C_\infty}$ and look at   
$V(\hat{a}_1, \hat{a}_2)$ in $\Spec(\sO_{\sX, C_\infty})$.  
As before, it follows by our construction that each irreducible  
component of $V(\hat{a}_1, \hat{a}_2)$ has codimension   
exactly one in $X_{\rm sing}\cap V(\hat{a}_1)$.  
We fix this $\hat{a}_2$ and the corresponding $a_2$.   
Again, $a_1,a_2, \ldots, a_{d-1}$ (with $a_2$ accordingly modified) form a  
regular sequence generating $I_{C, C_\infty}$, thanks to   
Lemma~\ref{lem:regularsequences}.  
  
In general, the choice of $\hat{a}_i$ depends on the previously chosen   
$\hat{a}_1,\ldots, \hat{a}_{i-1}$.  
It is chosen with the property that $\hat{a_i}$ is a unit at each   
generic point of $X_{\rm sing}\cap V(\hat{a}_1, \ldots, \hat{a}_{i-1})$, and   
that $\hat{a}_i$ lifts $a_i\in I_{C, C_\infty}$.  
This can be achieved, up to elements of $I_{C, C_\infty}^2$, since locally   
$V(\hat{a}_1, \ldots, \hat{a}_{i-1}) \cap C = C \not \supset    
V(\hat{a}_1, \ldots, \hat{a}_{i-1}) \cap X_{\rm sing}$.  
  
At the end of the process, we get   
$\hat{a}_1,\ldots, \hat{a}_{d-1} \in \sO_{\sX, C_\infty}$ with the following   
properties.  
\begin{enumerate}  
\item  
The sequence $\{\hat{a}_1,\ldots, \hat{a}_{d-1}\}$ restricts to a regular   
sequence $\{{a}_1,\ldots, {a}_{d-1}\}$ generating the ideal   
$I_{C, C_\infty}$ in   
$\sO_{X, C_\infty}$.   
The images of $\{{a}_1,\ldots, {a}_{d-1}\}$ in $I_{C, C_\infty}/ I_{C, C_\infty}^2$   
are the basis we started from.   
  
\item   
Let $V(\hat{a}_1,\ldots, \hat{a}_{d-1}) \subset \Spec(\sO_{\sX, C_\infty})$ be   
the closed subscheme of $\Spec(\sO_{\sX, C_\infty})$ defined by the ideal   
$(\hat{a}_1,\ldots, \hat{a}_{d-1})$. Then $V(\hat{a}_1,\ldots, \hat{a}_{d-1})$   
intersects $X_{\rm sing}$ in at most finitely many points   
(the intersection could be empty).   
  
\item   
Let $\hat{C}$ be the closure of  $V(\hat{a}_1,\ldots, \hat{a}_{d-1}) $ in $\sX$.   
Then $\hat{C}$ is flat over $S$ and there exists an open neighborhood $U$ of   
$C_\infty$ in $X$ such that $(\hat{C}\cap X)\cap U$  and $C\cap U$ coincide   
scheme-theoretically.  
In particular, if $T$ denotes the (finite) set of closed points of   
$\hat{C}\cap X_{\rm sing}$ together with the generic points of   
$\hat{C}\cap X$, then we have an isomorphism   
$\sO_{\hat{C}\cap X, T}\cong \cO_{C, C_\infty}\times R$, with $R$ an   
$1$-dimensional semi-local ring.   
\end{enumerate}  
  
Property (2) follows from the fact that, at each step, the generic points of   
$V(\hat{a}_1,\ldots, \hat{a}_{i})$ have height exactly one at each generic   
point of $X_{\rm sing}\cap V(\hat{a}_1, \ldots, \hat{a}_{i-1})$.  
Property (3) is clearly a consequence of (1) and of the construction.   
It tells us in particular that we can harmlessly throw  away any component of   
$\hat{C}$ which happens to be completely vertical (i.e., the   
structure map to $S$ factors through the closed point). This is  
because such a component has to be disjoint from $C$ in a neighborhood of   
$C_\infty$. Note that $\hat{C}$ can be taken with the reduced scheme structure,   
but it may not be integral even if $C$ is.  
  
It follows from (3) that the map on units   
$\cO_{\hat{C}, T}^\times \to \cO_{\hat{C}\cap X, T}^\times \times R$ is surjective.   
We can therefore find an element $\hat{f}$ in the ring of total quotients of   
$\hat{C}$ (which is by (3) a product of fields) which is a regular and   
invertible function in a neighborhood of $T$ and which restricts to   
$(f,1)$ (where $f$ was the given function on $C$).  
In particular, this implies that   
$\wt{\rho} ({\rm div}_{\hat{C}}(\hat{f}) )= {\rm div}_{C}(f)$.  
By ${\rm div}_{\hat{C}}(\hat{f})$, here we mean the sum of the divisors on the   
irreducible components of $\hat{C}$ if $\hat{C}$ is not integral.  
Note that ${\rm div}_{\hat{C}}(\hat{f})$ is an element of   
$\sZ_1^g(\sX)$ and that, we have  
${\rm div}_{\hat{C}}(\hat{f}) = \wt{\gamma}({\rm div}_C(f))$ by construction.   
Since we clearly have ${\rm div}_{\hat{C}}(\hat{f})=0$ in $\CH_1(\sX)_\Lambda$,   
this completes the proof of the proposition when $\nu: C \inj X$ is  
a regular closed immersion.  
  
We now prove the general case. So suppose we are given a good curve  
$\nu: C \to X$ and a rational function $f$ on $C$ as in the  
beginning of the proof of the proposition.  
By \cite[Lemma~3.5]{BK}, we can assume that the map $\nu: C \to X$  
is a complete intersection morphism.  
Now, we can find a commutative diagram  
\begin{equation}\label{eqn:Factor-gamma-*-0}   
\xymatrix@C.8pc{  
C \ar@{^{(}->}[r]^-{\nu'} \ar[dr]_-{\nu} & \P^M_X \ar[d]_-{q} \ar[r] &  
\P^M_{\sX} \ar[d]^-{q} \ar[dr] \\  
& X \ar[r] & \sX \ar[r] & S}  
\end{equation}  
for some $M \gg0 $ such that $\nu'$ is a regular closed embedding.  
Letting $\sY = \P^M_{\sX}$ and $Y = \P^M_X$, this gives a diagram  
  
\begin{equation}\label{eqn:Factor-gamma-*-1}   
\xymatrix@C.8pc{  
\sZ_0(Y, Y_{\rm sing})_{\Lambda} \ar[r]^-{{\wt{\gamma}}_Y} \ar[d]_-{q_*} &  
\CH_1(\sY)_{\Lambda} \ar[d]^-{q_*} \\  
\sZ_0(X, X_{\rm sing})_{\Lambda} \ar[r]^-{{\wt{\gamma}}_X} &   
\CH_1(\sX)_{\Lambda}.}  
\end{equation}  
  
Note that the push-forward map $q_*$ on the left is defined since  
$q$ is smooth (see \cite[Proposition~3.18]{BK}).  
It is easily seen from the construction of the cycle $Z_x$ associated  
to a regular closed point of $X$ in \propref{prop:rho-is-onto}  
that ~\eqref{eqn:Factor-gamma-*-1} commutes.  
We thus get   
\[  
\wt{\gamma}_X \circ \nu_*(\divf(f)) = \wt{\gamma}_X \circ q_*(\divf_C(f))  
= q_* \circ \wt{\gamma}_Y(\divf_C(f)) = 0.  
\]  
This finishes the proof of the proposition.  
\end{proof}

{\bf Proof of \thmref{thm:Main-1}:}  
The construction of $\wt{\rho}$ is given in   
~\eqref{eq:restriction-map-generators}. The existence of   
the map $\gamma$ such that ~\eqref{eqn:Main-1-0} commutes,   
follows directly from \corref{cor:Factor} and \propref{prop:Factor-gamma-*},  
using the fact that the surjection $\sZ_0(X, X_{\rm sing}) \surj \CH^{BK}_0(X)$  
factors as $\sZ_0(X, X_{\rm sing}) \surj \CH^{LW}_0(X) \surj \CH^{BK}_0(X)$.  
The surjectivity of $\gamma$ follows from \propref{prop:rho-is-onto}.  
$\hfill\square$

\subsection{}  
In the above notations, we have constructed a surjective group homomorphism   
\[  
\gamma \colon \CH^{LW}_0(X)_{\Lambda} \surj \CH_1(\sX)_\Lambda,  
\]  
which is (by construction) an inverse on the level of generators of   
the naive restriction map   
\[  
\wt{\rho} \colon \sZ_1^g(\sX)\to  \sZ_0(X, X_{\rm sing})  
\]  
for any regular projective and flat scheme $\sX$ over $S$  
without any assumption on the residue field  (apart from it being perfect).  
This also does not depend on the geometry of the reduced special fiber $X$.  
In particular, we can summarize what we have shown as follows.  
  
\begin{cor}\label{cor:sum}  
Let $A$ be a Henselian discrete valuation ring  with perfect residue field.  
Let $\sX$ be a regular scheme which is projective and flat over $A$ with  
reduced special fiber $X$. Suppose that the map  
$\wt{\rho} \colon \sZ_1^g(\sX)\to  \sZ_0(X, X_{\rm sing})$ descends to  
a morphism between the Chow groups  
\begin{equation}\label{eq:rho}  
\rho \colon \CH_1(\sX)_\Lambda \to \CH^{LW}_0(X)_{\Lambda}.  
\end{equation}  
Then $\rho$ is an isomorphism.  
\end{cor}

\section{The restriction isomorphism}\label{sec:Res-iso}  
We shall prove \thmref{thm:Main-2} in this section.  
In other words, we shall show that   
the restriction homomorphism $\rho$ of \eqref{eq:rho} does exist if additional assumption on the field $k$ or on the DVR $A$ hold. 
  
%is algebraically closed. We shall prove this by constructing a cycle class map  
%from the Levine-Weibel Chow group to {\'e}tale cohomology for singular schemes.  
%  
%After that, we will discuss some results over more general fields, by relating the existence of the restriction homomorphism $\rho$ to the validity of the Bloch-Quillen formula for zero cycles on singular varieties.  

\subsection{The \'etale cycle class map}\label{ssec:et-cycle-class}  
The cycle class map from the Chow groups to the {\'e}tale cohomology is  
well known for smooth schemes. More generally, the {\'e}tale realization  
of Voevodsky's motives tells us that there are such maps from the  
Friedlander-Voevodsky motivic cohomology of singular schemes to  
their {\'e}tale cohomology. But this is not good enough for us since we  
do not work with the motivic cohomology.  
In this section, we give a construction of the cycle class map from the  
Levine-Weibel Chow group of a singular scheme to its {\'e}tale  
cohomology using Gabber's Gysin maps \cite{Fujiwara}.

We let $k$ be a perfect field and let $X$ be an  
equi-dimensional quasi-projective scheme of dimension $d$ over $k$.   
Let $m$ be an integer prime to the exponential characteristic of $k$  
and let $\Lambda = {\Z}/{m\Z}$.  
  
Let $x\in X_{\rm reg}$ be a regular closed point of $X$.   
We have the sequence of maps   
\begin{equation}\label{eq:def-et-cycle-class}  
\Z \surj \Lambda \xrightarrow{\cong} H^0_{\et}(k(x), \Lambda)   
\stackrel{(1)}{\underset{\cong}\to}   
H^{2d}_{\{x\}, \et}(X, \Lambda(d)) \xrightarrow{(2)}   
H^{2d}_{\et}(X, \Lambda(d)).  
\end{equation}  
  
The arrow labeled $(1)$ is the Gysin map \cite{Gabber},   
using the fact that $x$ is a regular   
closed point of $X$. This is an  
isomorphism by the purity and excision theorems in {\'e}tale cohomology.   
The arrow labeled $(2)$ is the natural `forget support' map.   
Let $\delta_x$ denote the composite of all maps in   
~\eqref{eq:def-et-cycle-class}.  
We let $cyc^{\et}_X(x) = \delta_x(1)$ and  
extend it linearly to define a  group homomorphism  
\begin{equation}\label{eq:def-et-cycle-class-0}  
cyc^{\et}_X\colon \sZ_0(X, X_{\rm sing}) \to    
H^{2d}_{\et}(X, \Lambda(d)).  
\end{equation}  
  
We shall now show that this map factors through the modified Chow group  
$\CH^{BK}_0(X)$. It will then follow that it factors through $\CH^{LW}_0(X)$  
as well. So let $\nu \colon (C,Z)\to X$ be a good curve as  
in Definition~\ref{defn:0-cycle-S-1}. As in the proof of   
\propref{prop:Factor-gamma-*}, we can assume that $\nu$ is a  
local complete intersection morphism. In this case,   
Gabber's construction of push-forward map in {\'e}tale cohomology   
\cite{Fujiwara} gives us a push-forward map  
\[  
\nu_*\colon H^2_{\et}(C, \Lambda(1))\to H^{2d}_{\et}(X, \Lambda(d))  
\]  
and a diagram  
\begin{equation}\label{eq:commm-diag-push-etale}   
\xymatrix@C.8pc{   
\sZ_0(C,Z) \ar[r]^-{\nu_*} \ar[d]_-{cyc^{\et}_C} &   
\sZ_0(X, X_{\rm sing}) \ar[d]^-{cyc^{\et}_X} \\  
H^2_{\et}(C, \Lambda(1)) \ar[r]^-{\nu_*} & H^{2d}_{\et}(X, \Lambda(d)).}  
\end{equation}  
  
If $x \in C \setminus Z$ is a closed point so that $\nu(x) \in X_{\rm reg}$,  
the functoriality of the Gysin maps implies that  
the composite $H^0_{\et}(k(x), \Lambda(0)) \to H^2_{\et}(C, \Lambda(1))   
\xrightarrow{\nu_*} H^{2d}_{\et}(X, \Lambda(d))$ is the push-forward map  
associated to the finite complete intersection  
map $\Spec(k(x)) \to X$. Using this fact and  
the description   
\eqref{eq:def-et-cycle-class} of the cycle class map on generators,  
it follows that ~\eqref{eq:commm-diag-push-etale} is commutative.  
  
We can identify   
$\CH_0(C,Z)$ with $\Pic(C)$ according to \cite[Lemma 4.12]{BK}.  
The Kummer sequence then shows that  
there is a commutative diagram  
\[  
\xymatrix@C.8pc{  
\sZ_0(C,Z) \ar[rr]\ar[rd]_-{cyc^{\et}_C} &&   
\Pic(C) \ar[ld]\\  
& H^2_{\et}(C, \Lambda(1)).& }  
\]  
  
This immediately shows that for any rational function $f$ on $k(C)$ such that   
${\rm div}_C(f)\in \sR_0(C,Z)$, we have $cyc^{\et}_C({\rm div}_C(f))=0$ in   
$H^2_{\et}(C, \Lambda(1))$. But then, the commutativity of   
\eqref{eq:commm-diag-push-etale} proves that $\nu_*( {\rm div}_C(f))$ goes to   
zero in $H^{2d}_{\et}(X, \Lambda(d))$.   
We have therefore shown that the map $cyc^{\et}_X$ in  
~\eqref{eq:def-et-cycle-class-0} descends to a cycle class map on the  
Chow group:  
\begin{equation}\label{eq:def-et-cycle-class-1}  
cyc^{\et}_X: \CH^{BK}_0(X) \to H^{2d}_{\et}(X, \Lambda(d)).  
\end{equation}  
  
We shall denote its composite with the canonical surjection   
$\CH^{LW}_0(X) \surj \CH^{BK}_0(X)$ also by $cyc^{\et}_X$.

\subsection{The case of algebraically closed fields}\label{sec:Alg-closed}  
In the notations of \S~\ref{ssec:et-cycle-class}, suppose moreover that $k$ is   
separably (hence algebraically) closed and $X$ is projective over $k$.  
Write $X = \cup_{i=1}^n X_i$, where the $X_i$'s are the the irreducible   
components of $X$. In this case, we have a natural `trace' map   
\begin{equation}\label{eqn:trace}  
\tau_X \colon H^{2d}_{\et}(X, \Lambda(d)) \xrightarrow{\cong}  
\oplus_{i=1}^n H^{2d}_{\et}(X_i, \Lambda(d)) \xrightarrow{\cong}   
\oplus_{i=1}^n \Lambda.   
\end{equation}  
  
It follows by combining the exact sequence  
\[  
H^{2d-1}_{\et}(X_{\rm sing}, \Lambda(d)) \to  
H^{2d}_{c, \et}(X_{\rm reg}, \Lambda(d)) \to   
H^{2d}_{\et}(X, \Lambda(d)) \to H^{2d}_{\et}(X_{\rm sing}, \Lambda(d)),  
\]  
\cite[Chapter~VI, Lemma~11.3]{Milne} and the cohomological dimension  
bound $cd_{\Lambda}(X_{\rm sing}) \le 2d-2$ (as $k$ is separably closed)  
that the map $\tau_X$ in ~\eqref{eqn:trace} is an isomorphism.  
  
Note further that for any regular closed   
point $x\in X_{\rm reg}$, the composition  
\[  
\Lambda \xrightarrow{\cong}  H^0(k(x), \Lambda)   
\to H^{2d}_{\et}(X, \Lambda(d)) \xrightarrow{\cong}  
\oplus_{i=1}^n H^{2d}_{\et}(X_i, \Lambda(d)) \xrightarrow{\tau}   
\oplus_{i=1}^n \Lambda   
\]  
sends $1\in \Lambda$ to the element $1$ in the   
direct summand of $\oplus_{i=1}^n \Lambda$   
associated to the unique component of $X$ containing $x$ and to zero  
in all other summands.

Recall now from \cite[\S~1]{ESV} that there is a degree map    
${\rm deg}\colon \CH^{BK}_0(X)_{\Lambda} \to \bigoplus_{i=1}^n \Lambda$.  
This is considered in \textit{loc. cit.}   
for the Levine-Weibel Chow group, but the discussion there easily  
shows that it actually factors through the quotient   
$\CH^{BK}_0(X)$. This map is given by the sum of the degree maps for  
0-cycles on the irreducible components of $X$.   
In particular, for any regular closed point $x \in X$, the degree of $x$ is the   
element $1$ in the direct summand of $\oplus_{i=1}^n \Lambda$   
associated to the unique component of $X$ containing $x$ and is zero  
in all other summands.

Combining these two facts, we have a commutative diagram    
\begin{equation}\label{eq:diag-et-degree}  
\xymatrix@C.8pc{  
\CH^{BK}_0(X)_{\Lambda} \ar[rr]^-{\deg} \ar[rd]_-{cyc_{X}^{\et}} & &   
\oplus_{i=1}^n \Lambda  \\  
& H^{2d}_{\et}(X, \Lambda(d)). \ar[ru]_-{\tau_X}^-{\cong} & }  
\end{equation}  
  
In this setting, we have   
\begin{lem}\label{lem:et-iso}  
The degree map induces an isomorphism   
$\CH^{LW}_0(X)_\Lambda \xrightarrow{\cong} \oplus_{i=1}^n \Lambda$.   
In particular,  the {\'e}tale cycle class map   
$cyc_{X}^\et: \CH^{LW}_0(X)_\Lambda \to H^{2d}_{\et}(X, \Lambda(d))$   
is an isomorphism.  
\end{lem}  
\begin{proof}  
The second statement is a consequence of the first by   
\eqref{eq:diag-et-degree}.  
Since the degree map is clearly surjective (as $k$ is algebraically closed),   
it is enough to prove its injectivity.  
Since we are working with $\Z/m$-coefficients with $m\in k^\times$,   
it is in fact enough to prove that the subgroup $\CH_0(X)_{\rm deg =0}$ of   
0-cycles of degree zero is $m$-divisible.   
But this well known as $k$ is algebraically closed. Indeed, given any  
0-cycle $\alpha \in \sZ_0(X, X_{\rm sing})$ of degree zero, we can find  
a reduced Cartier curve $C \subset X$ which is regular along the support  
of $\alpha$. This implies that $\alpha$ lies in the image of the  
push-forward map $\Pic^0(C) \to \CH^{LW}_0(X)$. It is therefore enough to know  
that $\Pic^0(C)$ is $m$-divisible. But this is elementary.  
\end{proof}  
It is a straightforward exercise to deduce from the previous Lemma the isomorphism discussed in Remark \ref{rem:algclosed}.

\subsection{Results over non-algebraically closed fields}  
In this final section,  we suppose that the Gersten conjecture for Milnor $K$-theory holds for schemes over $A$.   
Thanks to \cite{Kerz09}, this is the case if $k\subset A$, i.e.,~if $A$ is an equicharacteristic DVR.   
  
Let $\Lambda = \Z/m\Z$, with $m$ prime to $p$, and let $n\geq 0$ be a non-negative integer.  
Recall (see e.g., \cite[8.2]{EKW}), that the $n$-th Milnor $K$-theory sheaf $\sK_{n, \Lambda}^M$ with $\Lambda$-coefficients is defined as the (Zariski or Nisnevich) sheafification of the presheaf on affine schemes sending an $A$-algebra $R$ to the quotient of   
$\Lambda\otimes_\Z T_n(R)$ by the two-sided ideal generated by elements of the form $a\otimes (1-a)$ with $a, 1-a\in R^\times$. Here $T_n(R)$ is the $n$-th tensor algebra of $R^\times$ (over $\Z$). Since in what follows we will only consider $\Lambda$-coefficients (unless explicitly mentioned), we drop it from the notation. Write $\sK_{n, Y}^M$ for the restriction of $\sK_{n, \Lambda}^M$ to the small (Zariski or Nisnevich) site of $Y$ for any $A$-scheme $Y$. If the residue field of $A$ is finite, we denote by the same symbol the sheaf of improved Milnor $K$-theory, with $\Lambda$ coefficients, in the sense of Kerz \cite{KerzImprovedMilnor}.

Let $\sX$ be again a regular scheme   
which is projective and flat over $A$, of relative dimension $d\geq 0$.   
One of the consequences of the Gersten conjecture is the so called Bloch formula, relating Milnor K-theory with the Chow groups. In particular, there is a canonical isomorphism  
\[ cyc^{M}_{\sX}\colon \CH_1(\sX)_\Lambda \xrightarrow{\simeq}  H^d(\sX_{\rm Nis}, \sK^M_{d, \sX})\]  
which is induced by the tautological ``cycle class map''  
\[cyc_{\sX}^M \colon \sZ_1(\sX)  =  \bigoplus_{x\in \sX_{(1)}} \Z \cong \bigoplus_{x\in \sX_{(1)}} K_0^M(k(x)),     \]  
where the right hand side appears as the last term of the Gersten resolution for $\sK_{d, \sX}^M$.  	   
  
Let now $X$ denote as before the reduced special fiber of $\sX$, and let $x\in X_{\rm reg}$ be a regular  closed point of $X$.   
We have a sequence of maps  
\[ \Lambda \cong K_0^M(k(x))\otimes_{\Z }\Lambda  \stackrel{(1)}{\underset{\cong}\to} H^d_{ \{x\}}(X_{\rm Zar}, \sK^M_{d, X}) \xrightarrow{(2)} H^d(X_{\rm Zar}, \sK^M_{d, X}) \xrightarrow{(3)} H^d(X_{\rm Nis}, \sK^M_{d, X}) \]  
where the isomorphism (1) follows from Kato's computation \cite[Theorem 2]{Kato} (using again the regularity of the point $x$), the map (2) is the canonical forget support map, and the map (3) is the change of topology from Zariski to the Nisnevich site.   
  
Extending this map linearly, we get a cycle class map  
\[cyc^M_{X}\colon \sZ_0(X, X_{\rm sing})\to H^d(X_{\rm Nis}, \sK^M_{d, X}).\]  
  
Recall now the following result from \cite{KG}  
\begin{thm}[Theorem 4.1, \cite{KG}] The cycle class map $cyc^M_{X}$ induces a surjective homomorphism  
	\[ cyc_{X}^M\colon \CH_0^{LW}(X) \twoheadrightarrow H^d(X_{\rm Nis}, \sK^M_{d, X}).\]  
	\end{thm}  
  
With this result at disposal, we can consider the following diagram.  
\begin{equation}\label{eq:restrictionMilnor}  
	\begin{tikzcd}  
		\sZ_1(\sX)^g_{\Lambda} \arrow[r] \arrow[d, "\tilde{\rho}"] & \CH_1(\sX)_\Lambda  \arrow{r}{cyc^M_{\sX}}[swap]{\cong}   \arrow[d, dashrightarrow, "\rho"]&  H^d(\sX_{\rm Nis}, \sK^M_{d, \sX}) \arrow[d, "\rho^M"]\\  
	\sZ_0(X, X_{\rm sing})_{\Lambda} \arrow[r] & \CH_0^{LW}(X)_{\Lambda}  \arrow[r, "cyc^M_X"] &  H^d(X_{\rm Nis}, \sK^M_{d, X})  
	\end{tikzcd}  
\end{equation}  
Note that the outer rectangle in \eqref{eq:restrictionMilnor} commutes, since the left vertical map is the restriction on cycles \eqref{eq:restriction-map-generators}, the right vertical map is the restriction homomorphism on Milnor $K$-theory and the composite horizontal maps are by definition the cycle class maps. We can therefore ask wether there exists a homomorphism making the left square commutative as well, i.e., if the map $\tilde{\rho}$ descends to a morphism between the Chow groups.   
  
This is clearly implied by the following Conjecture.  
\begin{conj}[Bloch-Quillen formula] The cycle class map $cyc^M_X$ is an isomorphism. \label{conj:Bloch-Quillen}  
\end{conj}  
If $X$ is regular, this is a well-known fact. It was originally proved by Bloch in \cite{BlochK2} for surfaces, and generalized by Kato \cite{Kato} in higher dimension. If $X$ is of dimension $1$, it can be interpreted as the chain of isomorphisms (with integral coefficients)  
\[ \CH_0^{LW}(X)\xrightarrow{\cong} \Pic(X) \cong H^1(X, \cO_X^\times) \]  
where the cohomology is taken with respect to the Zariski or the Nisnevich topology.  
  
For singular varieties of dimension $\geq 1	$, the status of this conjecture is summarized here.  
\begin{thm}\label{thm:BQ-true}Conjecture \ref{conj:Bloch-Quillen} is true in the following cases, with integral coefficients.   
	\begin{romanlist}  
	\item $X$ is a quasi-projective surface with isolated singularities, over any field $k$.  
	\item $X$ is a quasi-projective surface with arbitrary singularities.  
	\item $X$ is an affine surface over any perfect field.  
	\item $X$ is projective and regular in codimension $1$, over an algebraically closed field.  
	\item $X$ is quasi-projective with isolated singularities over a finite field.   
	\end{romanlist}  
\end{thm}  
Item i) was first verified by Pedrini and Weibel in \cite{PW}, and in the affine case by Levine and Weibel \cite{LW}. The case ii) is due to Levine \cite{LevineBloch} in the case of algebraically closed fields. A modification of Levine's argument can be used to extend the result to the case of an arbitrary (perfect) ground field, provided that one replaces the Levine-Weibel Chow group with its modified version introduced in \cite{BK}. This is done in \cite{BKS}.

The affine case iii) and the case of singularities in codimension at least 2 iv) are shown in \cite[Theorem 1.1 and 1.2]{KG} (the arguments are independent from the arguments used in \cite{BKS}), while case v) is \cite[Theorem 1.6]{KrishnaCFT}.  
Older results in the affine case where obtained by Barbieri-Viale in \cite{BV-Bloch}.  
We can now give another application of Theorem \ref{thm:Main-1}.  
\begin{cor}\label{cor:iso} Let $\sX$ and $A$ be as above. Then the restriction homomorphism $\tilde{\rho}$ of \eqref{eq:restriction-map-generators} factors through the rational equivalence classes if $k$ is finite and if $X$ has only isolated singularities, or if $\dim(X)=2$ (without restrictions on the type of singularities). In these cases, it induces an isomorphism  
	\[\rho\colon \CH_1(\sX)_{\Lambda} \xrightarrow{\simeq}  \CH_0^{LW}(X)_{\Lambda}.\]  
	If $k$ is finite, both groups are finite.  
	\begin{proof}It is an immediate consequence of the commutative diagram \eqref{eq:restrictionMilnor},  given Theorem  \ref{thm:BQ-true}. By Corollary \ref{cor:sum**}, the induced map $\rho$ is automatically an isomorphism. Finally, by \cite[Theorem 1.2]{KrishnaCFT}, the group $ \CH_0^{LW}(X)_{\Lambda}$ is finite if the residue field $k$ is finite.   
		\end{proof}  
\end{cor}  
  
%\begin{rem}Let $X$ be a quasi-projective surface over a field $k$. In  \cite{PW-div}, it is stated as Theorem A that there exists a chain of isomorphisms for every $n\in k^\times$  
%	\[\CH_0^{LW}(X)/n \cong H^2(X_{\rm Zar}, \sK_{2, X})/n \cong H^2(X_{\rm Zar}, \sK_{2, X}/n) \cong H^2(X_{\rm Zar}, \mathcal{H}^2(\mu_n^{\otimes 2})) \]  
%where the last term is the second cohomology group fo the Zariski sheaf associated to the presheaf $U\mapsto H^2_{\et}(U, \mu_n^{\otimes 2})$.  This result, without restrictions on the singularities of $X$ or on the ground field $k$ relies on an unpublished result of Levine \cite{Levine-unp}, about the existence of Chern classes in the singular case. Assuming this, we can remove the assumption that $X$ has isolated singularities in the two-dimensional case from the above Corollary, giving a complete description of $\CH_1(\sX)_\Lambda$ in the equicharacteristic case.   
%\end{rem}  

\vskip .3cm  
  
\noindent\emph{Acknowledgments.}   
This project started while the first-named author was visiting the Tata Institute of Fundamental Research  in November 2016, and the final part of this project was completed during the  
extended stay of the authors at the Hausdorff Research Institute for  
Mathematics (HIM), Bonn in 2017. The authors would like to thank both  
institutions for invitation and support. The authors would also like to thank  
H{\'e}l{\`e}ne Esnault, Moritz Kerz and Olivier Wittenberg for sending  
several valuable comments and suggestions on an earlier draft of this  
work, as well as the anonymous referee for their help in improving the exposition of the paper.

%\vskip .3cm  
  
%\enlargethispage{20pt}  

\end{document}